\documentclass[12pt,a4paper]{amsart}
\usepackage[marginratio=1:1]{geometry}

\usepackage[T1]{fontenc}
\usepackage[utf8]{inputenc}
\usepackage[american]{babel}
\usepackage[draft=false]{hyperref}
\usepackage{enumitem}
\usepackage{graphicx}

\usepackage{amsfonts}
\usepackage{amssymb}
\usepackage{tikz-cd}
\usepackage{float}
\usepackage{verbatim}
\makeatletter
\providecommand*{\twoheadrightarrowfill@}{%
	\arrowfill@\relbar\relbar\twoheadrightarrow
}
\providecommand*{\twoheadleftarrowfill@}{%
	\arrowfill@\twoheadleftarrow\relbar\relbar
}
\providecommand*{\xtwoheadrightarrow}[2][]{%
	\ext@arrow 0579\twoheadrightarrowfill@{#1}{#2}%
}
\providecommand*{\xtwoheadleftarrow}[2][]{%
	\ext@arrow 5097\twoheadleftarrowfill@{#1}{#2}%
}
\makeatother

\newcommand{\A}{{\mathcal A}}
\newcommand{\B}{{\mathfrak B}}

\newcommand{\C}{{\mathfrak C}}
\newcommand{\M}{{\mathcal M}}

\newcommand{\R}{{\mathbb{R}}}

\DeclareMathOperator{\tp}{{tp}}
\DeclareMathOperator{\Th}{{Th}}

\DeclareMathOperator{\cl}{{cl}}

\DeclareMathOperator{\id}{{id}}
\DeclareMathOperator{\aut}{{Aut}}

\DeclareMathOperator{\dcl}{{dcl}}
\DeclareMathOperator{\st}{{st}}
\DeclareMathOperator{\ext}{{ext}}
\DeclareMathOperator{\SL}{{SL}}
\DeclareMathOperator{\AUT}{{AUT}}
\DeclareMathOperator{\acl}{{acl}}

\newtheorem{thm}{Theorem}[section]
\newtheorem{conj}[thm]{Conjecture}
\newtheorem{ques}[thm]{Question}
\newtheorem{problem}[thm]{Problem}
\newtheorem{lem}[thm]{Lemma}
\newtheorem{fct}[thm]{Fact}
\newtheorem{cor}[thm]{Corollary}
\newtheorem{prop}[thm]{Proposition}

\theoremstyle{remark}
\newtheorem{rem}[thm]{Remark}
\theoremstyle{definition}
\newtheorem{dfn}[thm]{Definition}

\newtheorem*{clm*}{Claim}
\newtheorem{ex}[thm]{Example}
\newcounter{claimcounter}[thm]
\newenvironment{clm}{\stepcounter{claimcounter}{\noindent {\textbf{Claim}} \theclaimcounter:}}{}

\title{Definable topological dynamics}
\author{Krzysztof Krupi\'nski}
\email{kkrup@math.uni.wroc.pl}
\address{
	Instytut Matematyczny, Uniwersytet Wrocławski\\
	pl. Grunwaldzki 2/4\\
	50-384 Wrocław, Poland
}
\thanks{The author is supported by Narodowe Centrum Nauki grants 2012/07/B/ST1/03513 and 2015/19/B/ST1/01151}



\date{}

\begin{document}
	
	\begin{abstract}
		For a group $G$ definable in a first order structure $M$ we develop basic topological dynamics in the category of definable $G$-flows. In particular, we give a description of the universal definable $G$-ambit and of the semigroup operation on it. We find a natural epimorphism from the Ellis group of this flow to the definable Bohr compactification of $G$, that is to the quotient $G^*/{G^*}^{00}_M$ (where $G^*$ is the interpretation of $G$ in a monster model). 
More generally, we obtain these results locally, i.e. in the category of $\Delta$-definable $G$-flows for any fixed set $\Delta$ of formulas of an appropriate form. In particular, we define local connected components ${G^*}^{00}_{\Delta,M}$ and ${G^*}^{000}_{\Delta,M}$, and show that $G^*/{G^*}^{00}_{\Delta,M}$ is the $\Delta$-definable Bohr compactification of $G$. We also note that some deeper arguments from \cite{KrPi} can be adapted to our context, showing for example that our epimorphism from the Ellis group to the $\Delta$-definable Bohr compactification factors naturally yielding a continuous epimorphism from the $\Delta$-definable generalized Bohr compactification to the $\Delta$-definable Bohr compactification of $G$. Finally, we propose to view certain topological-dynamic and model-theoretic invariants as Polish structures which leads to some observations and questions.
	\end{abstract}

	\maketitle

	\section{Introduction}

Topological dynamics was introduced to model theory by Newelski in \cite{Ne1,Ne2} and then further developed by various authors, e.g. in \cite{GiPePi}, \cite{Ja}, \cite{YaLo}, \cite{ChSi}, \cite{KrPi} and \cite{KrPiRz}. There are several natural categories to develop topological dynamics in model theory. The most natural are the categories of definable and externally definable ``objects''. So far, however, mostly the externally definable category has been studied (the definable one was investigated under the extra assumption of definability of types, which makes both categories the same). In this paper, we develop basic topological dynamics in the category of definable flows, without the definability of types assumption.

Recall that a {\em $G$-flow} is a pair $(G,X)$, where $G$ is a group acting on a compact, Hausdorff space by homeomorphisms. We always consider discrete flows, i.e. with no topology on $G$ (or, if one prefers, with the discrete topology on $G$). A {\em $G$-ambit} is a $G$-flow $(G,X,x_0)$ with a distinguished point $x_0$ whose $G$-orbit is dense. With the obvious notion of a homomorphism of $G$-ambits, a universal $G$-ambit always exists and is unique; this universal ambit is exactly $(G,\beta G, e)$ (see \cite[Chapter 1, Proposition 2.6]{Gl}), where $\beta G$ is the Stone-$\check{\mbox{C}}$ech compactification of $G$.

Now, we recall some flows which have been investigated in model theory.

Let $G$ be a group $\emptyset$-definable in a first order structure $M$. By $S_G(M)$ we denote the space of complete types over $M$ containing the formula defining $G$; equivalently, this is the space of ultrafilters in the Boolean algebra of all definable (with parameters from $M$) subsets of $G$, equipped with the Stone topology. By $S_{G,\ext}(M)$ we denote the space of all externally definable complete types over $M$ containing $G$, that is the space of ultrafilters 
in the Boolean algebra of all externally definable subsets of $G$ (i.e. subsets which are intersections with $G$ of sets definable in arbitrary elementary extensions of $M$). 

Following \cite{GiPePi} and \cite{KrPi}, we will use the notion of [externally] definable $G$-flows. Namely, let $C$ be a compact, Hausdorff space. A map $f\colon G \to C$ is said to be {\em [externally] definable} if for all disjoint, closed subsets $C_1$ and $C_2$ of $C$ the preimages $f^{-1}[C_1]$ and $f^{-1}[C_2]$ can be separated by an [externally] definable subset of $G$. An {\em [externally] definable $G$-flow} is a $G$-flow $(G,X)$ such that for every $x \in X$ the map $f_x\colon G \to X$ defined by $f_x(g)=gx$ is [externally] definable.
An {\em [externally] definable $G$-ambit} is an [externally] definable $G$-flow $(G,X,x_0)$ with a distinguished point $x_0$ such that the orbit $Gx_0$ is dense in $X$.

Now, $G$ acts by translations as groups of homeomorphisms of the compact spaces $S_{G}(M)$, $S_{G,\ext}(M)$ and $\beta G$, turning them into $G$-flows. As mentioned before, $(G,\beta G, e)$ is the universal $G$-ambit. Similarly, $(G,S_{G,\ext}(M),\tp_{\ext}(e/M))$ is the universal externally definable $G$-ambit. In particular, by the universality, there is a left-continuous semigroup operation on $S_{G,\ext}(M)$ turning it into a semigroup isomorphic to the Ellis semigroup $E(S_{G,\ext}(M))$ (for the definitions of Ellis [semi]groups see Section \ref{section: preliminaries}). So one can consider both the minimal ideals and the Ellis group inside $S_{G,\ext}(M)$ instead of in $E(S_{G,\ext}(M))$.
However, the ambit $(G,S_G(M),\tp(e/M))$ is not necessarily definable (it is so if all types in $S_G(M)$ are definable in which case $S_G(M)=S_{G,\ext}(M)$), and we do not have a natural semigroup operation on $S_G(M)$. This makes $S_{G,\ext}(M)$ and the category of externally definable $G$-flows easier to work with, and that is why topological dynamics has been developed in this context. On the other hand,  $S_G(M)$ and definable flows are simpler and more natural objects from the point of view of model theory. This motivates our interest in the category of definable $G$-flows. Another, more concrete motivation stems from the fact that even if the language and the model $M$ are both countable, the universal externally definable $G$-ambit $S_{G,\ext}(M)$ may be ``big'' (e.g. not metrizable) which causes some difficulties in the application of topological dynamics to Borel cardinalities of bounded invariant equivalence relations \cite{KrPiRz}. In contrast, as we explain in this paper, under such a countability assumption, the universal definable $G$-ambit is always metrizable, which may lead to simplifications of some proofs concerning Borel cardinalities or even to new results. On the other hand, our research leads to interesting questions and relations with Polish structures introduced in \cite{Kr} and further developed by several authors, e.g. in \cite{KrWa}.

In this paper, we will consider even a more general category than that of definable $G$-flows, namely the category of $\Delta$-definable $G$-flows, where $\Delta$ is an arbitrary set of formulas $\delta(x;y,z,\bar t)$ of the form $\varphi(y\cdot x \cdot z, \bar t)$ which contains a formula defining $G$. The definitions of $\Delta$-formulas, the space $S_{G,\Delta}(M)$, and $\Delta$-definable $G$-flows are given in Section \ref{section: preliminaries}.

In Section \ref{section: definable ambit}, we present the universal $\Delta$-definable $G$-ambit as the quotient of $S_{G,\Delta}(M)$ by some closed equivalence relation $E_\Delta$, we describe the semigroup operation on $S_{G,\Delta}(M)/E_\Delta$, and, using it, we give a description of the relation $E_\Delta$. 

In Section \ref{section: connected components}, we define local versions of the connected components ${G^*}^{00}_M$ and ${G^*}^{000}_M$, namely ${G^*}^{00}_{\Delta,M}$ and ${G^*}^{000}_{\Delta,M}$, and we show that $G^*/{G^*}^{00}_{\Delta,M}$ is the $\Delta$-definable Bohr compactification of $G$. The proof follows the lines of the argument from \cite{GiPePi} showing that $G^*/{G^*}^{00}_{M}$ is the definable Bohr compactification, but it also requires some additional observations (in particular, the fact that  ${G^*}^{00}_{\Delta,M}$ is a normal subgroup is not completely obvious, and we need Lemma \ref{lemma: mu is contained} in order to show that the quotient map $G \to G^*/{G^*}^{00}_{\Delta,M}$ is $\Delta$-definable). Using results from Section \ref{section: definable ambit}, we find an explicitly given epimorphism $\theta$ from the universal $\Delta$-definable $G$-ambit $S_{G,\Delta}(M)/E_\Delta$ to $G^*/{G^*}^{00}_{\Delta,M}$, whose restriction to the Ellis group is also an epimorphism. We formulate interesting questions concerning an analogous statement for ${G^*}/{G^*}^{000}_{\Delta,M}$.

In Section \ref{section: results from KrPi}, we explain that some deeper arguments from \cite{KrPi} can be adapted to the $\Delta$-definable context, which results in Theorems \ref{theorem: gen. Bohr comp.}, \ref{theorem: factorization by H(uM)} and \ref{theorem: strong amenability}. We also use \cite[Theorem 5.6]{ChSi} to get a variant of this theorem in our definable context (see Corollary \ref{corollary: from ChSi}).

In Section \ref{section: Polish structures}, we explain how to treat various invariants (e.g. the Ellis group or the $\Delta$-definable generalized Bohr compactification of $G$) as Polish structures, which suggests that in some situations one could expect to get structural and topological theorems about these invariants via application of theorems on small Polish structures. We make a few observations in this direction and formulate some questions.

\section{Preliminaries}\label{section: preliminaries}

Detailed preliminaries concerning topological dynamics in model theory were given in several papers, so here we only recall a few things. For more details, see e.g. \cite[Section 1]{KrPi}. A good reference for classical topological dynamics is for example \cite{Au} or \cite{Gl}.

\begin{dfn}
The {\em Ellis semigroup} of the flow $(G,X)$, denoted by $E(X)$, is the closure of the collection of functions $\{\pi_g \colon g \in G\}$ (where $\pi_g\colon X \to X$ is given by 
$\pi_g(x)=gx$) in the space $X^X$ equipped with the product topology, with composition as semigroup operation.
\end{dfn}

This semigroup operation is continuous in the left coordinate, $E(X)$ is also a $G$-flow, and minimal subflows of $E(X)$ are exactly minimal left ideals with respect to the semigroup structure on $E(X)$, in particular they are cosed and so compact. The following was proved by Ellis (e.g. see \cite[Propositions 3.5, 3.6]{El} and \cite[Chapter 1, Propositions 2.3, 2.5]{Gl}).

\begin{fct}\label{Ellis theorem}
Let ${\M}$ be a minimal left ideal in $E(X)$, and let $J({\M})$ be the set of all idempotents in ${\M}$. Then:\\
i) For any $p \in {\M}$, $E(X)p={\M}p={\M}$.\\
ii) ${\M}$ is the disjoint union of sets $u{\M}$ with $u$ ranging over $J({\M})$.\\
iii) For each $u \in J({\M})$, $u{\M}$ is a group with the neutral element $u$, where the group operation is the restriction of the semigroup operation on $E(X)$.\\
iv) All the groups $u{\M}$ (for $u \in J({\M})$) are isomorphic, even when we vary the minimal ideal ${\M}$.
\end{fct}

For a given group $G$ we say that a $G$-ambit $(G,X,x_0)$ is {\em universal} if for every $G$-ambit $(G,Y,y_0)$ there exists a (unique) homomorphism $h \colon X \to Y$ of $G$-flows mapping $x_0$ to $y_0$. The following fact is fundamental \cite[Chapter 1, Proposition 2.6]{Gl}.

\begin{fct}\label{universality of beta G}
$(G,\beta G, e)$ is the unique up to isomorphism universal $G$-ambit.
\end{fct}

Using this fact, one gets an ``action'' of $\beta G$ on any $G$-flow $(G,X)$, namely for $x \in X$ there is a unique flow homomorphism $h_x\colon (G,\beta G, e) \to (G,X,x)$, and for $p \in \beta G$ we define $px=h_x(p)$.  More explicitly, this action is given by $px=\lim g_ix$ for any net $(g_i)$ of elements of $G$ converging to $p$ in $\beta G$. In particular, $\beta G$ acts on itself, and denoting this action by $*$, one has $(p*q)x=p(qx)$ for all $p,q \in \beta G$ and $x \in X$. In particular, $*$ is a semigroup operation on $\beta G$ which is continuous on the left and whose restriction to $G \times G$ is the original  group operation on $G$. One easily checks that $(\beta G,*) \cong E(\beta G)$ (by sending $p \in \beta G$ to the function $(x \mapsto px) \in E(\beta G)$). In particular, Fact \ref{Ellis theorem} applies to $(\beta G, *)$ in place of $E(\beta G)$.\\

In this paper, $M$ denotes a model of a first order theory in a language $\mathcal L$, and $\C \succ M$ is a monster model; $G$ will be a group $\emptyset$-definable in $M$, and $G^*$ its interpretation in $\C$. Group multiplication, denoted by $\cdot$, will be often skipped for simplicity, but sometimes we will write it explicitly.\\ 

As above, a {\em universal [externally] definable $G$-ambit} is defined as an [externally] definable $G$-ambit which maps homomorphically (by a (unique) homomorphism of $G$-ambits) to an arbitrary [externally] definable $G$-ambit. It is clear (by general category theory reasons) that, in each of these two categories, if a universal $G$-ambit exists, then it is unique up to isomorphism. 
As mentioned in the introduction, $(G,S_{G,\ext}(M),\tp_{\ext}(e/M))$ is the unique up to isomorphism universal externally definable $G$-ambit, so, in the same way as above, we get a semigroup operation on $S_{G,\ext}(M)$. The existence of the universal definable $G$-ambit is justified in Remark \ref{remark: existence of universal definable G-ambit}. The main point of the current paper is to describe the universal definable $G$-ambit as well as its local versions.

Let $\Delta$ be a subset of the set of all formulas (without parameters) in the language $\mathcal L$ in the variable $x$ of the same sort as $G$ and some parametric variables. By a {\em $\Delta$-formula} over $M$ we mean 
a Boolean combination of instances of formulas from $\Delta$ with parameters from $M$. By $S_{G,\Delta}(M)$ we denote the compact space of complete $\Delta$-types over $M$ concentrating on $G$, equivalently, the space of ultrafilters of relatively $\Delta$-definable over $M$ subsets of $G$. In the case when $\Delta$ is the collection of all formulas in the variable $x$ of the same sort as $G$ and arbitrary parametric variables, $\Delta$-formulas over $M$ are just all formulas over $M$ in the variable $x$, and $S_{G,\Delta}(M)=S_G(M)$. 

Sometimes we will be using $\Delta$-types over sets of parameters other than $M$. In such situations, all the definitions of ``$\Delta$-objects'' are analogous, except we take the convention that by a $\Delta$-formula over $A$ we mean a formula over $\dcl(A)$ which is equivalent to a Boolean combination of instances of formulas from $\Delta$ with parameters from $\dcl(A)$ (we do it in order to avoid writing ``$\dcl$'' many times in the paper). Another option is to define $\Delta$-formulas over $A$ as those formulas over $A$ which are equivalent to Boolean combinations of instances of formulas from $\Delta$ with arbitrary parameters (but then some statements would be slightly less general).

We naturally extend the definition of a definable function from $G$ to a compact space $C$ to a $\Delta$-definable context.

\begin{dfn}\label{definable functions}
Let $C$ be a compact space. A map $f\colon G \to C$ is {\em $\Delta$-definable} if for all disjoint, closed subsets $C_1$ and $C_2$ of $C$ the preimages $f^{-1}[C_1]$ and $f^{-1}[C_2]$ can be separated by a $\Delta$-definable subset of $G$.
\end{dfn}

Now, we recall the definition of a definable map defined on the monster model and we extend it to local versions.

\begin{dfn}
Let $C$ be a compact space.\\
i) A function $f\colon G^* \to C$ is {\em $M$-definable} if for every closed $C_1 \subseteq C$ the preimage $f^{-1}[C_1]$ is type-definable over $M$.\\
ii) A function $f\colon G^* \to C$ is {\em $\Delta$-definable over $M$} if for every closed $C_1 \subseteq C$ the preimage $f^{-1}[C_1]$ is $\Delta$-type-definable over $M$ (i.e. the intersection of sets defined by $\Delta$-formulas over $M$).
\end{dfn}

The first item of the next fact is \cite[Lemma 3.2]{GiPePi}, and the second one (generalizing the first one) can be proved analogously. As usual, $C$ denotes a compact, Hausdorff space.

\begin{lem}\label{prolongation}
1) Definable context:\\
i) If $f\colon G \to C$ is definable, then it extends uniquely to an $M$-definable function $f^*\colon G^* \to C$. Moreover, $f^*$ is given by the formula $\{f^*(a)\}=\bigcap_{\varphi \in \tp(a/M)} \cl(f[\varphi(M)])$.\\
ii) Conversely, if $f^*\colon G^* \to C$ is an $M$-definable function, then $f^*|_G\colon G \to C$ is definable.\\
2) $\Delta$-definable context:\\
i) If $f\colon G \to C$ is $\Delta$-definable, then it extends uniquely to a function $f^*\colon G^* \to C$ which is $\Delta$-definable over $M$. Moreover, $f^*$ is given by the formula $\{f^*(a)\}=\bigcap_{\varphi \in \tp_\Delta(a/M)} \cl(f[\varphi(M)])$.\\
ii) Conversely, if a function $f^*\colon G^* \to C$ is a $\Delta$-definable over $M$, then $f^*|_G\colon G \to C$ is $\Delta$-definable.
\end{lem}

\begin{rem}\label{rem: prolongation}
If $f\colon G \to C$ is $\Delta$-definable, then the unique function  $f^*\colon G^* \to C$ which extends $f$ and is $\Delta$-definable over $M$ is also given by 
$\{f^*(a)\}=\bigcap_{\varphi \in \tp(a/M)} \cl(f[\varphi(M)])$.
\end{rem}

\begin{proof}
This follows from the formula in point 2)(i) of Lemma \ref{prolongation} and the observations that $\bigcap_{\varphi \in \tp(a/M)} \cl(f[\varphi(M)])$ is non-empty (which follows from the compactness of $C$) and contained in $\bigcap_{\varphi \in \tp_\Delta(a/M)} \cl(f[\varphi(M)])$.
\end{proof}

The following remark follows easily from definitions.

\begin{rem}\label{remark: factorization}
A function $f\colon G^* \to C$ is $\Delta$-definable over $M$ if and only if there is a continuous function $h: S_{G,\Delta}(M) \to C$ such that $f=h\circ r$, where $r\colon G^* \to S_{G,\Delta}(M)$ is the obvious map $a \mapsto \tp_\Delta(a/M)$. In particular, a function $f\colon G^* \to C$ is $M$-definable if and only if there is a continuous function $h: S_{G}(M) \to C$ such that $f=h\circ r$, where $r\colon G^* \to S_{G}(M)$ is the obvious map.
\end{rem}

We extend the notion of the definable flow in a natural way. Namely, a flow $(G,X)$ will be called {\em $\Delta$-definable} if for every $x \in X$ the map $f_x\colon G \to X$ given by $f_x(g)=gx$ is $\Delta$-definable.

The first part of the following observation was made in \cite[Remark 1.12]{KrPi}, and the generalization to the second part can be obtained analogously.

\begin{rem}\label{definability of products}
i) A product of definable $G$-flows is a definable $G$-flow.\\
ii) A product of $\Delta$-definable $G$-flows is a $\Delta$-definable $G$-flow.
\end{rem}

Whenever we consider the quotient of $G^*$ by a bounded, $M$-invariant equivalence relation $E$, we can equip it with the {\em logic topology} which is defined by saying that a subset of the quotient is closed if its preimage in $G^*$ is type-definable (equivalently, type-definable over $M$). If $E$ is type-definable, then $G^*/E$ is a compact, Hausdorff space; if $E$ is only invariant, then $G^*/E$ is only quasi-compact. The definition of the logic topology applies in particular to the quotient $G^*/H$, where $H$ is an arbitrary bounded index, $M$-invariant subgroup of $G^*$. If $H$ is type-definable, then $G^*/H$ is a compact, Hausdorff group. For details on the logic topology see e.g. \cite{KrNe} and \cite[Section 2]{Pi1}. The following basic remark will be useful later.

\begin{rem}\label{remark: denseness very basic}
Let $E$ be an bounded, $M$-invariant equivalence relation on $G^*$. Then $G/E$ is dense in $G^*/E$
\end{rem}

\begin{proof}
Consider any non-empty, open subset $U$ of $G^*/E$, and let $\pi \colon G^* \to G^*/E$ be the quotient map. Then $\pi^{-1}[U]$ is a non-empty, $\bigvee$-definable over $M$ subset of $G^*$, and as such it has a point in $G$. 
\end{proof}

\section{Universal $\Delta$-definable $G$-ambit}\label{section: definable ambit}

From now on, we fix a set $\Delta$ of formulas about which we assume that:

\begin{enumerate}
\item It consists of some formulas $\delta(x;y,z,\bar t)$ of the form $\psi(y\cdot x \cdot z; \bar t)$.
\item The formula $G(y \cdot x\cdot z)$ is in $\Delta$, where $G(x)$ is a formula over $\emptyset$ which defines $G$.
\end{enumerate}

This makes sure that any left or right translate of a $\Delta$-formula by an element of $G$ is still a $\Delta$-formula. Note that if all formulas as in (1) are included in $\Delta$, then being ``a subset of $G$ which is $\Delta$-definable over $M$'' is the same thing as being ``a subset of $G$ which is definable over $M$'', so we are in the definable context.


\begin{rem}\label{remark: existence of universal definable G-ambit}
There exists a unique (up to isomorphism) universal $\Delta$-definable $G$-ambit. 
\end{rem}

\begin{proof}
Uniqueness is clear. To show existence, consider a set $(G,X_i,x_i)_{i \in I}$ of all up to isomorphism $\Delta$-definable $G$-ambits (such a set exists, as there is a common bound on the cardinalities of all $G$-ambits). Now, let $X'=\prod_{i} X_i$, $x=(x_i)_i$, and let $X$ be the closure of the orbit of $x$ under the coordinatewise action of $G$. Then $(G,X')$ is $\Delta$-definable by Remark  \ref{definability of products}, and so $(G,X,x)$ is a $\Delta$-definable $G$-ambit. From the construction, we see that $(G,X,x)$ maps on any $\Delta$-definable $G$-ambit, i.e. it is universal.
\end{proof}

By the assumption on the form of the formulas in $\Delta$, we see that $G$ acts on $S_{G,\Delta}(M)$ by 
$$g \tp_{\Delta}(a/M)=\tp_{\Delta}(ga/M)$$
so that $(G,S_{G,\Delta}(M), \tp_{\Delta}(e/M))$ is a $G$-ambit.

Let $(G,{\mathcal U}, x)$ be the universal $\Delta$-definable $G$-ambit. Then, $f_x \colon G \to {\mathcal U}$ given by $f_x(g)=gx$ is $\Delta$-definable. Thus, by Fact \ref{prolongation}, it extends to the function $f^*_x \colon G^* \to {\mathcal U}$ which is $\Delta$-definable over $M$. Hence, by Remark \ref{remark: factorization}, there exists a continuous function $h\colon S_{G,\Delta}(M) \to {\mathcal U}$ such that $f^*_x=h \circ r$, where $r \colon G^* \to S_{G,\Delta}(M)$ is the obvious map (namely $a \mapsto \tp_\Delta(a/M)$). The function $h$ is uniquely determined by $f^*_x$, and we see that $h(\tp_{\Delta}(e/M))=x$. 

\begin{rem}\label{remark: homomorphism of ambits}
The above map $h$ yields a homomorphism from the $G$-ambit $(G,S_{G,\Delta}(M), \tp_{\Delta}(e/M))$ to $(G,{\mathcal U}, x)$.
\end{rem}

\begin{proof}
It remains to show that $h(gp)=gh(p)$ for any $g \in G$ and $p \in S_{G,\Delta}(M)$. By the explicit formula for $f^*_x$, we have
$$\{gh(p)\}=g\bigcap_{\varphi \in p} \cl(f_x[\varphi(M)])=\bigcap_{\varphi \in p} \cl(f_x[g\varphi(M)]) = \bigcap_{\psi \in gp} \cl(f_x[\psi(M)])=h(gp).$$
\end{proof}

If $(G,S_{G,\Delta}(M), \tp_{\Delta}(e/M))$ was a $\Delta$-definable $G$-ambit, we could proceed with the development of the theory exactly as in the well-understood externally definable context. However, in general, this ambit is not necessarily $\Delta$-definable (see Example \ref{example: E non-trivial}). Recall that if all types in $S_G(M)$ are definable, then $(G,S_{G}(M), \tp(e/M))$ is definable and coincides with $(G,S_{G,\ext}(M),\tp_{\ext}(e/M))$, and the categories of definable and externally definable $G$-flows coincide. Our goal is to start to develop a theory of definable $G$-flows without the definability of types assumption. 

Let $E_\Delta$ be the equivalence relation on $S_{G,\Delta}(M)$ given by
$$E_\Delta(p,q) \iff h(p)=h(q).$$
In the definable context (i.e. when  $\Delta$ consists of all formulas of the appropriate form), we will write $E$ instead of $E_\Delta$.


\begin{cor}\label{cor: definable ambit is a quotient}
$(G,S_{G,\Delta}(M)/E_\Delta,\tp_{\Delta}(e/M)/E_\Delta)$ is the universal $\Delta$-definable $G$-ambit.
\end{cor}

\begin{proof}
By Remark \ref{remark: homomorphism of ambits}, $E_\Delta$ is closed and invariant under the action of $G$. This implies that $(G,S_{G,\Delta}(M)/E_\Delta,\tp_{\Delta}(e/M)/E_\Delta)$ naturally becomes a  $G$-ambit. From the very definition of $E_\Delta$, we get that $h$ factors through the quotient map $S_{G,\Delta}(M) \to S_{G,\Delta}(M)/E_\Delta$ yielding an isomorphism from the $G$-ambit $(G,S_{G,\Delta}(M)/E_\Delta,\tp_{\Delta}(e/M)/E_\Delta)$ to $(G,{\mathcal U}, x)$. This completes the proof as $(G,{\mathcal U}, x)$ was chosen to be the universal $\Delta$-definable $G$-ambit.
\end{proof}

\begin{rem}\label{remark: uniqueness of E}
i) $E_\Delta$ is the unique equivalence relation $F$ on $S_{G,\Delta}(M)$ for which $(G,S_{G,\Delta}(M)/F,\tp_{\Delta}(e/M)/F)$, with the action of $G$ defined by $g(p/F):=(gp)/F$, is the universal $\Delta$-definable $G$-ambit.\\
ii) In the definable context (i.e. when  $\Delta$ consists of all formulas of the appropriate form), if all types in $S_G(M)$ are definable, then $E$ is trivial (i.e. it is the equality).
\end{rem}

\begin{proof}
i) Suppose $F_1$ and $F_2$ are two such relations. It is enough to show that $F_1 \subseteq F_2$. By the universality of $(G,S_{G,\Delta}(M)/F_1,\tp_{\Delta}(e/M)/F_1)$, there is $f\colon S_{G,\Delta}(M)/F_1 \to S_{G,\Delta}(M)/F_2$ which is a homomorphism of $G$-ambits. Let $f'\colon S_{G,\Delta}(M) \to S_{G,\Delta}(M)/F_2$ be the composition of $f$ with the quotient map $S_{G,\Delta}(M) \to S_{G,\Delta}(M)/F_1$. We see that $f'$ is a homomorphism from the $G$-ambit  $(G,S_{G,\Delta}(M), \tp_{\Delta}(e/M))$ to $(G,S_{G,\Delta}(M)/F_2, \tp_{\Delta}(e/M)/F_2)$. But there is only one such homomorphism, and it is given by $p \mapsto p/F_2$. Therefore, $f(p/F_1)=p/F_2$ for any $p$, hence $F_1 \subseteq F_2$.\\[1mm]
ii) By assumption, $(G,S_G(M), \tp(e/M))$ is a definable $G$-ambit. Hence, Corollary \ref{cor: definable ambit is a quotient} implies that $(G,S_G(M), \tp(e/M))$ is the universal definable $G$-ambit (literally, the quotient $S_G(M)/E$ is universal definable, but this implies the universal property of $S_G(M)$). And we finish using (i).
\end{proof}

Define an equivalence relation $E_\Delta'$ on $G^*$ by
$$E_\Delta'(a,b) \iff E_\Delta(\tp_{\Delta}(a/M),\tp_{\Delta}(b/M)).$$
In the definable context, we will skip $\Delta$ and write $E'$.
We see that $E_\Delta'$ is a type-definable over $M$, bounded equivalence relation on $G^*$  which is coarser than the relation of having the same complete $\Delta$-type over $M$. We have a natural topological identification of $S_{G,\Delta}(M)/E_\Delta$ with $G^*/E_\Delta'$, which we will be using freely.

Our main goal is to give an explicit description of the relation $E_\Delta$, equivalently of $E_\Delta'$.

By Corollary \ref{cor: definable ambit is a quotient}, a standard argument (a sketch of which we give below for the reader's convenience) shows that there is a left continuous semigroup operation $*$ on $S_G(M)/E_\Delta$ which is given by
\begin{equation}
p/E_\Delta * q/E_\Delta = \lim_{g \rightarrow p/E_\Delta} g(q/E_\Delta),\label{eq: 0}
\end{equation}
equivalently,
$$a/E_\Delta' * b/E_\Delta'= \lim_{g \rightarrow a/E_\Delta'} g(a/E_\Delta').$$
(When we write ``$g$ tends to $p/E_\Delta$'', we mean ``$\tp_\Delta(g/M)/E_\Delta$ tends to $p/E_\Delta$ with $g$ ranging over $G$'', i.e. $\lim_{g \rightarrow p/E_\Delta} g(q/E_\Delta)=r/E_\Delta$ if and only if for every open neighborhood $U$ of $r/E_\Delta$ there is an open neighborhood $V$ of $p/E_\Delta$ such that for all $g \in G$ such that $\tp_\Delta(g/M)/E_\Delta \in V$ one has $g(q/E_\Delta) \in U$. And analogously in the equivalent definition.)

\begin{proof}[Sketch of the proof]
To simply notation, denote $(G,S_{G,\Delta}(M)/E_\Delta,\tp_{\Delta}(e/M)/E_\Delta)$ by $(G,\mathcal U,x_0)$. By the universality of $(G,{\mathcal U},x_0)$, for every $x \in {\mathcal U}$ there is a unique $f_x \colon {\mathcal U} \to {\mathcal U}$ which is a $G$-flow homomorphism mapping $x_0$ to $x$. For $p \in {\mathcal U}$ put 
$$p * x = f_x(p).$$
From the choice of $f_x$, the following properties follow immediately.
\begin{enumerate}
\item[(i)] $*$ is continuous in the left coordinate.
\item[(ii)] $*$ extends the action of $G$, i.e. $(gx_0) * x = gx$ for all $g \in G$ and $x \in {\mathcal U}$.
\item[(iii)] $\lim_{g \rightarrow p} gq = p * q$ for all $p,q \in {\mathcal U}$ (here $g \rightarrow p$ means that $gx_0$ tends to $p$).
\end{enumerate}
Finally, we leave as an exercise (using nets and limits) to check that $*$ is associative.
\end{proof}

Next, we give an explicit formula for $*$, which is similar to the one in the externally definable case.

\begin{prop}\label{proposition: formula for *}
For any $p,q \in S_{G,\Delta}(M)$, $p/E_\Delta * q/E_\Delta=\tp_\Delta(a\cdot b/M)/E_\Delta$, where $b \models q$ and $a$ realizes a $\Delta$-coheir extension of $p$ over $M,b$ (i.e. $\tp_\Delta(a/M,b)$ is finitely satisfiable in $M$).
\end{prop}

\begin{proof}
Note that a basis of open neighborhoods of $\tp_{\Delta}(ab/M)/E_\Delta \in S_{G,\Delta}(M)/E_\Delta$ consists of  the sets 

$$U_\varphi:=\{ r/E_\Delta : [r]_{E_\Delta} \subseteq [\varphi]\}$$ 
with $\varphi$ ranging over all $\Delta$-formulas over $M$ such that $\tp_{\Delta}(ab/M)/E_\Delta \in U_\varphi$, where $[r]_{E_\Delta}=\{ q \in S_{G,\Delta}(M): E_\Delta(r,q)\}$ and $[\varphi]=\{ q \in S_{G,\Delta}(M) : \varphi \in q\}$. Indeed, first of all, each set $U_{\varphi}$ is open in the quotient topology, because the preimage of the complement of $U_{\varphi}$ under the quotient map consists of all the types $r \in S_{G,\Delta}(M)$ for which there exists a type $r' \in S_{G,\Delta}(M)$ such that $E_\Delta(r,r')$ and $\neg \varphi(x) \in r'$ (and so we see that this preimage is closed as a projection of a closed set in a product of compact, Hausdorff spaces). Secondly, take any open neighborhood $U$ of $\tp_{\Delta}(ab/M)/E_\Delta$. Then $[\tp_{\Delta}(ab/M)]_{E_\Delta}$ is contained in the preimage $U'$ of $U$ under the quotient map. Since $[\tp_{\Delta}(ab/M)]_{E_\Delta}$ is closed and $U'$ is open, by the compactness of $S_{G,\Delta}(M)$, we can find a $\Delta$-formula (over $M$) $\varphi(x)$ such that $[\tp_{\Delta}(ab/M)]_{E_\Delta} \subseteq [\varphi(x)] \subseteq U'$, and so $U_\varphi$ is an open neighborhood of  $\tp_{\Delta}(ab/M)/E_\Delta$ contained in $U$.

Consider any $\varphi$ as above. Then $\tp_{\Delta}(ab/M) \vdash E_\Delta'(ab,x) \vdash \varphi(x)$. So there is $\psi(x) \in \tp_{\Delta}(ab/M)$ such that

\begin{equation}
(\exists y)(\psi(y) \wedge E_\Delta'(x,y)) \vdash \varphi(x).\label{eq: 1}
\end{equation}
Clearly $\models \psi(ab)$.

Consider any $\delta(w) \in p$. As $\models \delta(a) \wedge \psi(ab)$, we get $\delta(w) \wedge \psi(wb) \in \tp_\Delta(a/M,b)$. (Note that in order to have that $\psi(wb)$ is a $\Delta$-formula over $M,b$, we use our convention which allows Boolean combinations of instances of formulas from $\Delta$ with parameters from $\dcl(M,b)$; namely, we have to use parameters which are products of the form $b\cdot g \in \dcl(M,b)$ where $g \in G$.) By the assumption that $\tp_{\Delta}(a/M,b)$ is finitely satisfiable in $M$, there exists $g_{\delta,\varphi} \in G$ such that $\models \delta(g_{\delta,\varphi}) \wedge \psi(g_{\delta,\varphi} \cdot b)$. By (\ref{eq: 1}), we conclude that 
$E_\Delta'(g_{\delta,\varphi} \cdot b, x) \vdash \varphi(x)$, and so
\begin{equation}
g_{\delta,\varphi} (q/E_\Delta)=\tp_\Delta(g_{\delta,\varphi} \cdot b/M)/E_\Delta \in U_\varphi.\label{eq: 2} 
\end{equation}

The collection of all formulas $\varphi$ as above forms a directed set (with $\varphi_1\leq \varphi_2$ iff $\varphi_2 \vdash \varphi_1$) and similarly the collection of all $\delta$'s from $p$ forms a directed set; the product of these two directed sets is also a directed set with the product preorder, and the limits below are computed with respect to this product preorder. 

By (\ref{eq: 2}), we conclude that

\begin{equation}
\lim_{\delta,\varphi} g_{\delta,\varphi}  (q/E_\Delta)=\tp_{\Delta}(ab/M)/E_\Delta.\label{eq: 3}
\end{equation}

On the other hand,
$$\lim_{\delta,\varphi} \tp_{\Delta}(g_{\delta,\varphi}/M) = p,$$
so
$$\lim_{\delta,\varphi} \tp_{\Delta}(g_{\delta,\varphi}/M)/E_\Delta = p/E_\Delta,$$ 
which by virtue of (\ref{eq: 0}) implies that

\begin{equation}
\lim_{\delta,\varphi} g_{\delta,\varphi} (q/E_\Delta) = p/E_\Delta * q/E_\Delta. \label{eq: 4}
\end{equation}

From (\ref{eq: 3}) and (\ref{eq: 4}), we get $p/E_\Delta * q/E_\Delta = \tp_{\Delta}(ab/M)/E_\Delta$.
\end{proof}

Define  a relation ${F_0}_\Delta$ on $G^*$ as follows. ${F_0}_\Delta(\alpha,\beta)$ holds if there exist $a, b, a_1, b_1 \in G^*$ such that the following conditions hold:
\begin{enumerate}
\item $\alpha=a\cdot b$ and $\beta=a_1\cdot b_1$,
\item $\tp_{\Delta}(a/M)=\tp_{\Delta}(a_1/M)$ and $\tp_{\Delta}(b/M)=\tp_{\Delta}(b_1/M)$,
\item $\tp_{\Delta}(a/M,b)$ and $\tp_{\Delta}(a_1/M,b_1)$ are both finitely satisfiable in $M$. 
\end{enumerate}

The relation ${F_0}_\Delta$ is clearly reflexive and symmetric, but there is no obvious reason why it should be transitive. Let $F_\Delta$ be the transitive closure of ${F_0}_\Delta$. We see that ${F_0}_\Delta$ and $F_\Delta$ are $M$-invariant. Finally, let ${\bar F}_\Delta$ be the finest type-definable over $M$ equivalence relation on $G^*$ containing the relation $F_\Delta$. We easily see that all the relations ${F_0}_\Delta$, $F_\Delta$ and ${\bar F}_\Delta$ are coarser than the relation of having the same complete $\Delta$-type over $M$, so the last two are bounded equivalence relations.

As usual, in the definable context (i.e. when $\Delta$ consists of all formulas of the appropriate form), we skip the index $\Delta$ and write $F_0$, $F$ and $\bar F$.

\begin{rem}
The relations ${F_0}_\Delta$, $F_\Delta$ and ${\bar F}_\Delta$ are all invariant under the action of $G$, so $(G,G^*/{\bar F}_\Delta,e/{\bar F}_\Delta)$ is a $G$-ambit.
\end{rem}

\begin{proof}
First, we check that ${F_0}_\Delta$ is invariant under $G$. Consider any $g \in G$ and $\alpha,\beta \in G^*$ such that ${F_0}_\Delta(\alpha,\beta)$. Take $a,b,a_1,b_1$ from the definition of ${F_0}_\Delta$ witnessing that ${F_0}_\Delta(\alpha,\beta)$ holds. Since $g \in G \subseteq M$, one easily checks that $ga,b,ga_1,b_1$ witness that ${F_0}_\Delta(g\alpha,g\beta)$ holds.

Since ${F_0}_\Delta$ is invariant under $G$, so is $F_\Delta$. For any $g \in G$, the equivalence relation $g{\bar F}_\Delta$ is type-definable over $M$ and contains $gF_\Delta=F_\Delta$, so $\bar F_\Delta \subseteq  g \bar F_\Delta$. This implies that ${\bar F}_\Delta$ is invariant under $G$, which yields the natural action of $G$ on $G^*/{\bar F}_\Delta$ given by $g(h/{\bar F}_\Delta):= (gh)/{\bar F}_\Delta$. By the definition of the logic topology, we see that this is an action by homeomorphisms. Finally, the $G$-orbit of $e/{\bar F}_\Delta$ equals $G/{\bar F}_\Delta$, which is dense in the logic topology 
by Remark \ref{remark: denseness very basic}.
\end{proof}

We will also need the following general remark.

\begin{rem}\label{remark: added because of the Referee}
If $D$ is a type-definable subset of $G^*$ which is a union of sets of realizations of complete $\Delta$-types over $M$, then $D$ is $\Delta$-type-definable over $M$.
\end{rem}

\begin{proof}
Let $\pi \colon G^* \to S_{G,\Delta}(M)$ be the map given by $\pi(a):= \tp_{\Delta}(a/M)$, and let $p(x)$ be a partial type defining $D$. Then $\pi[D]$ is the subset of $S_{G,\Delta}(M)$ consisting off all types consistent with $p(x)$, so $\pi[D]$ is closed, i.e. it is the set of all types in $S_{G,\Delta}(M)$ extending some partial $\Delta$-type $p'$ over $M$. Since $D$ is a union of sets of realizations of complete $\Delta$-types over $M$, we see that $D=\pi^{-1}[\pi[D]]$, so $D=p'(G^*)$ is $\Delta$-type-definable over $M$.
\end{proof}

\begin{thm}\label{theorem: main theorem}
${\bar F}_\Delta=E_\Delta'$, so $(G,G^*/{\bar F}_\Delta, e/{\bar F}_\Delta)$ is the universal $\Delta$-definable $G$-ambit.
\end{thm}

\begin{proof}
First, we prove that ${\bar F}_\Delta \subseteq E_\Delta'$.  For this, it is enough to show that ${F_0}_\Delta \subseteq E_\Delta'$, because then clearly $F_\Delta \subseteq E_\Delta'$ which implies that ${\bar F}_\Delta \subseteq E_\Delta'$ by the definition of ${\bar F}_\Delta$ and the fact that $E_\Delta'$ is type-definable over $M$.

So, consider any $\alpha, \beta \in G^*$ such that ${F_0}_\Delta(\alpha,\beta)$. Then we have $a,b,a_1,b_1$ satisfying (1), (2) and (3) from the definition of ${F_0}_\Delta$. Let $p=\tp_{\Delta}(a/M)$ and $q=\tp_{\Delta}(b/M)$. By Proposition \ref{proposition: formula for *}, we get
$ \tp_{\Delta}(\alpha/M)/E_\Delta=\tp_\Delta(a\cdot b/M)/E_\Delta=p/E_\Delta*q/E_\Delta=\tp_{\Delta}(a_1\cdot b_1/M)/E_\Delta=\tp_{\Delta}(\beta/M)/E_\Delta.$
Hence, $E_\Delta'(\alpha,\beta)$.

Now, we will prove that $E_\Delta'\subseteq {\bar F}_\Delta$. Note that it is enough to show that $(G,G^*/{\bar F}_\Delta,e/{\bar F}_\Delta)$ is a $\Delta$-definable $G$-ambit. Indeed, if we know that $(G,G^*/{\bar F}_\Delta,e/{\bar F}_\Delta)$ is $\Delta$-definable, then, by the universality of $(G,G^*/E_\Delta',e/E_\Delta')$ (see Corollary \ref{cor: definable ambit is a quotient}), 
there exists a $G$-flow homomorphism $\sigma_1\colon G^*/E_\Delta' \to G^*/{\bar F}_\Delta$ such that $\sigma_1(g/E_\Delta')=g/{\bar F}_\Delta$ for all $g \in G$. On the other hand, by the already proven fact that ${\bar F}_\Delta \subseteq E_\Delta'$, there is a $G$-flow homomorphisms $\sigma_2 \colon G^*/{\bar F}_\Delta \to G^*/E_\Delta'$ given by $\sigma_2(a/{\bar F}_\Delta):=a/E_\Delta'$. 
Therefore, $(\sigma_1 \circ \sigma_2)|_{G/{\bar F}_\Delta} = \id_{G/{\bar F}_\Delta}$, hence $\sigma_1 \circ \sigma_2 =\id_{G^*/{\bar F}_\Delta}$, so $\sigma_2$ is injective which implies that ${\bar F}_\Delta=E_\Delta'$.

So, our goal is to prove that  $(G,G^*/{\bar F}_\Delta,e/{\bar F}_\Delta)$ is $\Delta$-definable. This means that for any $a \in G^*$ the function $f_a \colon G \to G^*/{\bar F}_\Delta$ given by $f_a(g)=(g \cdot a)/{\bar F}_\Delta$ is $\Delta$-definable. Fix any $a \in G^*$.

Define ${\bar f}_a \colon G^* \to G^*/{\bar F}_\Delta$ by ${\bar f}_a(\alpha)=(\alpha'\cdot a)/{\bar F}_\Delta$ for any (equivalently, some) $\alpha' \in G^*$ such that $\tp_{\Delta}(\alpha'/M)=\tp_{\Delta}(\alpha/M)$ and $\tp_{\Delta}(\alpha'/M,a)$ is finitely satisfiable in $M$. By the definition of ${\bar F}_\Delta$, this function is well-defined. We see that ${\bar f}_a$ extends $f_a$. Hence, it remains to show that ${\bar f}_a$ is $\Delta$-definable over $M$.

So, consider any $D \subseteq G^*/{\bar F}_\Delta$ which is closed. Let $\pi \colon G^* \to G^*/{\bar F}_\Delta$ be the quotient map. Then $\pi^{-1}[D]$ is type-definable over $M$ by a partial type $p(x)$.
We see that ${\bar f}_a^{-1}[D]$ is the set of all $\alpha \in G^*$ for which there exists $\alpha' \in G^*$ such that 
$$\tp_{\Delta}(\alpha'/M)=\tp_\Delta(\alpha/M) \; \mbox{and}\; \tp_\Delta(\alpha'/M,a)\; \mbox{finitely satisfiable in}\; M\; \mbox{and}\; \models p(\alpha'a).$$
%
Hence,  ${\bar f}_a^{-1}[D]$ is type-definable (over $M \cup\{a\}$) and it is also a union of sets of realizations of complete $\Delta$-types over $M$. Therefore, it is $\Delta$-type-definable over $M$ by Remark \ref{remark: added because of the Referee}.
\end{proof}

\begin{ques}
1) Is ${F_0}_\Delta$ type-definable?\\
2) Is $F_\Delta$ generated by ${F_0}_\Delta$ in finitely many steps?\\
3) Is $\bar{F}_\Delta$ equal to $F_\Delta$?
\end{ques}

The answers are probably negative in general and the problem is to find appropriate counter-examples. It would be also interesting to understand when the answers are positive. Note that they are trivially positive when we work in the definable context and all types in $S_G(M)$ are definable, as then $F_0(\alpha,\beta) \iff \alpha \equiv_M \beta$.

\section{Connected components and $\Delta$-definable Bohr compactification}\label{section: connected components}

A very important aspect of the topological-dynamic approach to model theory have been connections of some topological-dynamic invariants (e.g. the Ellis group) with model-theoretic invariants (such as quotients by various connected components of groups). Here, we try to say something about such connections in our $\Delta$-definable context. In particular, we introduce two $\Delta$-definable connected components and relate one of them to the Ellis group of the universal $\Delta$-definable ambit. We also give a desciption of the $\Delta$-definable Bohr compactification.

Recall that:
\begin{itemize}
\item ${G^*}^{00}_M$ is the smallest $M$-type-definable subgroup of $G^*$ of bounded index,
\item ${G^*}^{000}_M$ is the smallest $M$-invariant subgroup of $G^*$ of bounded index.
\end{itemize}

Take $\Delta$ as in the previous section. We extend the first definition to the local context in the following way.


\begin{dfn}
${G^*}^{00}_{\Delta,M}$ is the smallest $\Delta$-type-definable over $M$ subgroup of $G^*$ of bounded index.
\end{dfn}

Note that $G^*$ is $\Delta$-definable over $M$ (even over $\emptyset$) by the formula defining $G$, so ${G^*}^{00}_{\Delta,M}$ exists as the intersection of all $\Delta$-type-definable over $M$ subgroups of $G^*$ of bounded index.

Another definition of ${G^*}^{00}_\Delta$ has been proposed by E. Hrushovski in his lecture notes on approximate equivalence relations, but the above definition is more appropriate in our current situation of $\Delta$-definable topological dynamics. Later, we will also propose a local version of ${G^*}^{000}_M$, which leads to some questions.

It is well-known that ${G^*}^{000}_{M} \leq {G^*}^{00}_M$ are both normal subgroups of $G^*$ (e.g. see \cite[Lemma 2.2]{Gi}). 

\begin{prop}\label{proposition: G00 is normal}
${G^*}^{00}_{\Delta,M}$ is a normal subgroup of $G^*$.
\end{prop}

\begin{proof}

Note that ${G^*}^{00}_{\Delta,M}$ is normalized by $G$, which follows from the fact that the conjugate of a $\Delta$-formula over $M$ by any element of $G$ remains a $\Delta$-formula over $M$. Therefore, 
$${G^*}^{00}_{\Delta,M}=\bigcap_{\varphi(x) \in \A} \varphi(G^*)$$
for some family $\A$ of formulas over $M$ which is closed under conjugations by elements of $G$.

Let $\{ g _i: i \in I\}$ be a bounded set of representatives of right cosets of ${G^*}^{00}_{\Delta,M}$ in $G^*$. Then clearly
$$\bigcap_{g \in G^*} ({G^*}^{00}_{\Delta,M})^g= \bigcap_{i \in I} ({G^*}^{00}_{\Delta,M})^{g_i}.$$
So this intersection is invariant over $M$ and also type-definable, and hence it is type-definable over $M$. Thus,
$$\bigcap_{g \in G^*} ({G^*}^{00}_{\Delta,M})^g = \bigcap_{\varphi(x) \in \B} \varphi(G^*)$$
for some family $\B$ of formulas over $M$. 

It is enough to show that ${G^*}^{00}_{\Delta,M} \subseteq \bigcap_{g \in G^*} ({G^*}^{00}_{\Delta,M})^g$. So take any $\varphi(x) \in \B$. Then there are $\varphi_1(x),\dots,\varphi_n(x) \in \A$ and $a_1,\dots, a_n \in \{g_i: i \in I\}$ such that $\varphi_1(G^*)^{a_1} \cap \dots \cap \varphi_n(G^*)^{a_n} \subseteq \varphi(G^*)$. Since $M \prec \C$, there are $h_1,\dots,h_n \in G$ for which $\varphi_1(G^*)^{h_1} \cap \dots \cap \varphi_n(G^*)^{h_n} \subseteq \varphi(G^*)$. Since $\A$ is closed under conjugations by elements of $G$, we get ${G^*}^{00}_{\Delta,M} \subseteq  \varphi_1(G^*)^{h_1} \cap \dots \cap \varphi_n(G^*)^{h_n}$. We conclude that ${G^*}^{00}_{\Delta,M} \subseteq \varphi(G^*)$, and the proof is complete.
\end{proof}

Let $\mu$ be the subgroup of $G^*$ generated by all elements of the form $a^{-1}b$ for $\tp_{\Delta}(a/M)=\tp_{\Delta}(b/M) \in S_{G,\Delta}(M)$. Since there are only boundedly many complete $\Delta$-types over $M$, we get

\begin{rem}\label{rem: mu of bounded index}
$\mu$ has bounded index in $G^*$.
\end{rem}

\begin{lem}\label{lemma: mu is contained}
$\mu \leq  {G^*}^{00}_{\Delta,M}$.
\end{lem}

\begin{proof}
By compactness and the fact that ${G^*}^{00}_{\Delta,M}$ is a group $\Delta$-type-definable over $M$, we can present ${G^*}^{00}_{\Delta,M}$ as the intersection of some family $\{ \varphi_i(G^*)\}_{i \in I}$ of sets $\Delta$-definable over $M$
such that for every $i \in I$:
\begin{enumerate}
\item $\varphi_i(G^*)$ is symmetric, i.e. $e \in \varphi_i(G^*)=\varphi_i(G^*)^{-1}$,
\item there is $j \in I$ for which $\varphi_j(G^*)\cdot \varphi_j(G^*) \subseteq \varphi_i(G^*)$.
\end{enumerate}

Since ${G^*}^{00}_{\Delta,M}$ is of bounded index, each $\varphi_i(G^*)$ is left generic, i.e. there exist $g_{i,1},\dots,g_{i,n_i} \in G$ such that $g_{i,1}\varphi(G^*) \cup \dots \cup g_{i,n_i}\varphi_i(G^*)=G^*$.

We need to show that if $\tp_{\Delta}(a/M)=\tp_{\Delta}(b/M)$, then $a^{-1}b \in {G^*}^{00}_{\Delta,M}$. For this it is enough to show that for any $i \in I$, $a^{-1}b \in \varphi_i(G^*)$. Choose $j \in I$ for which $\varphi_j(G^*)\cdot \varphi_j(G^*) \subseteq \varphi_i(G^*)$. Since each $g_{j,k}\varphi_j(x)$ is a $\Delta$-formula over $M$, we get that $g_{j,k}\varphi_j(x) \in \tp_{\Delta}(a/M)$ for some $k \in \{1,\dots,n_j\}$. But then $g_{j,k}\varphi_j(x) \in \tp_{\Delta}(b/M)$. Hence, $a^{-1}b \in \varphi_j(G^*)^{-1}\varphi_j(G^*)=\varphi_j(G^*)\varphi_j(G^*) \subseteq \varphi_i(G^*)$.
\end{proof}

The quotients $G^*/{G^*}^{00}_M$ and $G^*/{G^*}^{00}_{\Delta,M}$ will be always equipped with the logic topology (see the end of Section \ref{section: preliminaries}).

Recall that a {\em definable compactification} of $G$ is a definable homomorphism from $G$ to a compact, Hausdorff group with dense image. The {\em definable Bohr compactification} of $G$ is a unique (up to $\cong$) {\em universal definable compactification} $G$, that is a definable compactification $f \colon G \to H$ such that for an arbitrary definable compactification $f_1 \colon G \to H_1$ there is a unique morphism from $f$ to $f_1$ (i.e. a group homomorphism $h\colon H \to H_1$ such that $h\circ f = f_1$). 

In \cite[Proposition 3.4]{GiPePi}, it was proven that the natural homomorphism from $G$ to $G^*/{G^*}^{00}_M$ is the definable Bohr compactification of $G$. Here, we extend this result to the $\Delta$-definable context. 

By a {\em $\Delta$-definable compactification of $G$} we mean a $\Delta$-definable
homomorphism from $G$ to a compact, Hausdorff group with dense image. A universal $\Delta$-definable compactification of $G$ is defined analogously to the universal definable compactification.


\begin{prop}\label{proposition: Delta-definable Bohr compactification}
The natural homomorphism from $G$ to $G^*/{G^*}^{00}_{\Delta,M}$ is a unique (up to $\cong$) universal $\Delta$-definable compactification of $G$, which we call the $\Delta$-definable Bohr compactification of $G$.
\end{prop}

\begin{proof}
First, we check that the natural homomorphism $\pi \colon G \to G^*/{G^*}^{00}_{\Delta,M}$ (given by $\pi(g)=g/{G^*}^{00}_{\Delta,M}$) is a $\Delta$-definable compactification of $G$. Density of $\pi[G]$ in $G^*/{G^*}^{00}_{\Delta,M}$ (equipped with the logic topology) follows from Remark \ref{remark: denseness very basic}. Let $\bar \pi\colon G^* \to G^*/{G^*}^{00}_{\Delta,M}$ be the quotient map.
For $\Delta$-definability of $\pi$ it is enough to show that $\bar \pi$ is $\Delta$-definable over $M$ (as $\bar \pi$ extends $\pi$ and we can use Lemma \ref{prolongation}). Consider any closed  $D \subseteq  G^*/{G^*}^{00}_{\Delta,M}$. Then, ${\bar \pi}^{-1}[D]$ is type-definable over $M$, and, by Lemma \ref{lemma: mu is contained}, it is a union of sets of realizations of complete $\Delta$-types over $M$. By Remark \ref{remark: added because of the Referee}, this implies that ${\bar \pi}^{-1}[D]$ is $\Delta$-type-definable over $M$.

Now, we check universality of $\pi$. We adapt the proof of \cite[Proposition 3.4]{GiPePi}.  Consider any $\Delta$-definable compactification $f\colon G \to C$. By Lemma \ref{prolongation}, there is a unique 
$\Delta$-definable over $M$ function $f^*\colon G^* \to C$ extending $f$. 
We check that $f^*$ is a homomorphism.
Consider any $a,b \in G^*$, and let $p:=\tp(a/M)$, $q:= \tp(b/M)$, and $r:=\tp(ab/M)$. Then, by Remark \ref{rem: prolongation} and compactness,

$$
\begin{array}{lll}
\{f^*(ab)\} &= &\bigcap_{\theta \in r} \cl (f[\theta(G)]) \subseteq \bigcap_{\varphi \in p,\psi \in q} \cl(f[\varphi(G) \cdot \psi(G)])\\
 &= &\bigcap_{\varphi \in p, \psi \in q} \cl(f[\varphi(G)]) \cdot \cl(f[\psi(G)]) \\
&= &\bigcap_{\varphi \in p} \cl(f[\varphi(G)])\cdot \bigcap_{\psi \in q} \cl(f[\psi(G)])\\
& = &\{f^*(a)f^*(b)\}.
\end{array}$$

Since $f^*$ is $\Delta$-definable over $M$, $\ker (f^*)$ is a normal subgroup of bounded index which is an intersection of some sets $\Delta$-type-definable over $M$. Since ${G^*}^{00}_{\Delta,M}$ is the smallest such group, we finish as in the proof of \cite[Proposition 3.4]{GiPePi}. Namely, there is a natural continuous homomorphism from $G^*/{G^*}^{00}_{\Delta,M}$ to $G^*/\ker(f^*)$, and $G^*/\ker(f^*)$ is naturally topologically isomorphic with $C$, so we get a continuous homomorphism  $h \colon G^*/{G^*}^{00}_{\Delta,M} \to C$ 
such that $h \circ \pi = f$. 
\end{proof}

\begin{rem}
i) ${G^*}^{00}_{\Delta,M}$ is the smallest type-definable over $M$ subgroup of $G^*$ containing $\mu$.\\
ii) ${G^*}^{00}_{\Delta,M} \geq\mu \cdot {G^*}^{00}_{M}$.
\end{rem}

\begin{proof}
i) Let $H$ be this smallest subgroup. Then $H$ is type-definable over $M$ and it is a union of sets of realizations of complete $\Delta$-types over $M$, so it is $\Delta$-type-definable over $M$ by Remark \ref{remark: added because of the Referee}. Thus, by Remark \ref{rem: mu of bounded index}, ${G^*}^{00}_{\Delta,M} \leq H$. The opposite inclusion follows  from Lemma \ref{lemma: mu is contained}.\\
ii) follows from Lemma \ref{lemma: mu is contained} and the definitions of ${G^*}^{00}_{\Delta,M}$ and ${G^*}^{00}_{M}$.
\end{proof}

This suggests the following generalization of the component ${G^*}^{000}_M$ to the local context.

\begin{dfn}
${G^*}^{000}_{\Delta,M}$ is the smallest normal, invariant over $M$ subgroup of $G^*$ of bounded index containing $\mu$.
\end{dfn}

We know that ${G^*}^{000}_M$ is generated by all elements $a^{-1}b$ for $a\equiv_M b$. Therefore, ${G^*}^{000}_M \leq \mu$, and so $\mu$ has bounded index in $G^*$ (which was already noted in Remark \ref{rem: mu of bounded index}). So, by the obvious observation that $\mu$ is invariant over $M$, we get

\begin{rem}
${G^*}^{000}_{\Delta,M}=\langle \mu^{G^*} \rangle$, where $\langle \mu^{G^*} \rangle$ denotes the normal closure of $\mu$.
\end{rem}

It is clear that when we are in the definable context, i.e. $\Delta$ consists of all formulas of the appropriate form, then ${G^*}^{00}_{\Delta,M}= {G^*}^{00}_M$ and ${G^*}^{000}_{\Delta,M}={G^*}^{000}_M=\mu$.

\begin{ques}
Is it true that ${G^*}^{000}_{\Delta,M} = \mu$? Equivalently, is $\mu$ a normal subgroup of $G^*$?
\end{ques}

Take the notation from Section \ref{section: definable ambit}. 
Now, we define a counterpart of Newelski's map defined in the externally definable context (see \cite[Proposition 4.4]{Ne1} or \cite[Proposition 3.1]{KrPi}).

\begin{prop}\label{proposition: theta is onto}
The map $\hat{\theta} \colon G^*/E'_{\Delta} \to G^*/{G^*}^{00}_{\Delta,M}$ given by $\hat{\theta}(a/E'_{\Delta})=a/{G^*}^{00}_{\Delta,M}$ is a well-defined, continuous semigroup epimorphism. 
\end{prop}

\begin{proof}
First, we check that $\hat{\theta}$ is well-defined.
Theorem \ref{theorem: main theorem} tells us that $E'_\Delta ={\bar F}_\Delta$. So, by the definition of ${\bar F}_\Delta$ and the fact that the relation of lying in the same left coset of ${G^*}^{00}_{\Delta,M}$ is type-definable over $M$, we see that it is enough to show that whenever ${F_0}_{\Delta}(\alpha,\beta)$, then $\beta ^{-1}\alpha\in {G^*}^{00}_{\Delta,M}$.

So, take any $\alpha, \beta$ with  ${F_0}_{\Delta}(\alpha,\beta)$. Then $\alpha=ab$ and $\beta= a_1b_1$ for some $a,b,a_1,b_1$ such that  $\tp_{\Delta}(a/M)=\tp_{\Delta}(a_1/M)$ and $\tp_{\Delta}(b/M)=\tp_{\Delta}(b_1/M)$. Then $\beta^{-1}\alpha=b_1^{-1}a_1^{-1}ab \in b_1^{-1} \mu b =b_1^{-1}b\mu^b \subseteq \mu \cdot \mu^b$. The last set is contained in ${G^*}^{00}_{\Delta,M}$, because ${G^*}^{00}_{\Delta,M}$ is normal and contains $\mu$ (by Proposition \ref{proposition: G00 is normal} and Lemma \ref{lemma: mu is contained}).

The fact that $\hat{\theta}$ is onto is clear from the definition of $\hat{\theta}$. Continuity follows from the definition of the logic topology. It remains to check that $\hat{\theta}$ is a homomorphism. Identifying $G^*/E'_{\Delta}$ with $S_{G,\Delta}(M)/E_\Delta$, we see that $\hat{\theta}(p/E_\Delta)=a/{G^*}^{00}_{\Delta,M}$ for any $a\models p$. Consider any $p,q \in S_{G,\Delta}(M)$ and choose $b \models q$ and $a$ satisfying a $\Delta$-coheir extension of $p$ over $M,b$. By Proposition \ref{proposition: formula for *}, $p/E_\Delta * q/E_\Delta=\tp_{\Delta}(ab/M)$. Thus, $\hat{\theta}(p/E_\Delta * q/E_\Delta)=ab/{G^*}^{00}_{\Delta,M}=\hat{\theta}(p/E_\Delta) \cdot \hat{\theta}(q/E_\Delta)$.
\end{proof}

Let $\M$ be a minimal left ideal in $G^*/E'_\Delta$ and $u$ an idempotent in $\M$. Let $\theta \colon u{\mathcal M} \to G^*/{G^*}^{00}_{\Delta,M}$ be the restriction of $\hat{\theta}$ to $u\M$. By Proposition \ref{proposition: theta is onto} and the fact that $u\M=u * G^*/E'_\Delta * u$, we get that $\theta$ is a group epimorphism.

In the externally definable case, a very important ingredient of the theory was the fact that there is also an epimorphism from the universal externally definable $G$-ambit and from its Ellis group to $G^*/{G^*}^{000}_M$. In the current context, we leave it as an open problem. 

\begin{problem}
1) Is $\hat{f} \colon G^*/E'_{\Delta} \to G^*/{G^*}^{000}_{\Delta,M}$ given by $\hat{f}(a/E'_{\Delta})=a/{G^*}^{000}_{\Delta,M}$ a well-defined semigroup epimorphism? Notice that whenever it is well-defined, then it is an epimorphism (as in the proof of Proposition \ref{proposition: theta is onto}).\\
2) If the answer to the above question in general is no, the problem is to understand when $\hat{f}$ is well-defined. This may lead to a new dividing line (motivated by topological dynamics and model theory together) in the class of groups definable in first order structures.\\
3) If $\hat{f}$ is well-defined, then its restriction $f$ to the Ellis group $u\M$ is also an epimorphism. If, however, it turns out that $f$ is not always well-defined, an interesting question is whether always there exits a (natural) epimorphism from $u\M$ to $G^*/{G^*}^{000}_{\Delta,M}$.\\
4) It is very interesting to consider the above questions in the definable context, i.e. when $\Delta$ consists of all formulas of the appropriate form. For example, the first question asks if $\hat{f} \colon G^*/E'  \to G^*/{G^*}^{000}_{M}$ given by $\hat{f}(a/E')=a/{G^*}^{000}_{M}$ is well-defined. 
\end{problem}

Let us look at the definable context. Notice that any counter-example to the statement that  $\hat{f} \colon G^*/E'  \to G^*/{G^*}^{000}_{M}$ given by $\hat{f}(a/E')=a/{G^*}^{000}_{M}$ is well-defined must satisfy ${G^*}^{000}_{M} \ne {G^*}^{00}_{M}$ and not all types in $S_G(M)$ are definable (the later property follows from Remark \ref{remark: uniqueness of E}(ii) and the fact that $a \equiv_M b$ implies $a^{-1}b \in {G^*}^{000}_M$). So a (simplest) natural candidate is the universal cover of $\SL_2(\mathbb{R})$ interpreted in the model $((\mathbb{Z},+),(K,+,\cdot))$, where $K$ is the real closure of the rationals. (For a model-theoretic analysis of the universal cover of $\SL_2(\R)$ (in particular, for the proofs that the two connected components differ) see \cite{CoPi} and \cite{GiKr}). However, an analysis of this example from the point of view of the definable topological dynamics seems quite complicated. Below we describe what happens in a much simpler example, namely in the unit circle $S^1(K)$.

The following remark follows easily from definitions by an argument as in the second paragraph of the proof of Proposition \ref{proposition: theta is onto}.
\begin{rem}
$\hat{f}^- \colon G^*/F_\Delta \to G^*/{G^*}^{000}_{\Delta,M}$ given by $\hat{f}^- (a/F_\Delta)=a/{G^*}^{000}_{\Delta,M}$ is a well-defined function. So, if $F_\Delta={\bar F}_\Delta$ (which is unlikely in general), then $\hat{f}$ is well-defined.
\end{rem}

\begin{ex}\label{example: E non-trivial}
Let $M:=(K,+,\cdot) $ be the real closure of  the rationals, $\Delta$ consist of all formulas of the appropriate form, and let $G:=S^1(K)$ be the unit circle computed in $K \times K$. Then 
$$S_{G}(M)=\{p_a: a\in S^{1}(\R) \setminus S^1(K)\} \cup \{p_a^-,p_a^+,q_a: a \in S^1(K)\},$$
where $q_a:=\tp(a/M)$ is the algebraic type isolated by $x=a$, $p_a$ is the cut at $a$, and $p_a^-$ and $p_a^+$ are the left and right cuts at $a$, respectively. 

It is well-known that $\mu={G^*}^{000}_M={G^*}^{00}_M$ is the subgroup of all infinitesimal elements (i.e. the monad of $1$ in $S^1(\C)$). Moreover, $G^*/{G^*}^{00}_M$ is homeomorphic with the real circle $S^1$.

We will show that for any $\alpha,\beta \in S^1(\C)$ 
\begin{equation}\label{equation: new in example}
F_0(\alpha,\beta) \iff (\beta \cdot \alpha^{-1} \in \mu \wedge \alpha \notin S^1(K) \wedge \beta \notin S^1(K)) \lor \alpha=\beta.
\end{equation}
$(\Rightarrow)$ First, consider the case when $\alpha \in S^1(K)$ or $\beta \in S^{1}(K)$. Without loss $\alpha \in S^1(K)$. Take $a,b,a_1,b_1$ witnessing that $F_0(\alpha,\beta)$. Then $a=\alpha \cdot b^{-1}$ and $\tp(a/K,b)$ is finitely satisfiable in $K$, so $b \in S^1(K)$, and hence $a \in S^1(K)$. Since $\tp(a_1/K)= \tp(a/K)$ and $\tp(b_1/K)=\tp(b/K)$, we conclude that $a_1=a$ and $b_1=b$. Therefore, $\alpha = \beta$.

Now, consider the case when  $\alpha \notin S^1(K)$ and $\beta \notin S^{1}(K)$. Take  $a,b,a_1,b_1$ witnessing that $F_0(\alpha,\beta)$. Then the computation from the second paragraph of the proof of Proposition \ref{proposition: theta is onto} shows that $\beta\cdot \alpha^{-1} \in \mu$.\\[1mm] 
$(\Leftarrow)$ If $\alpha =\beta$, then clearly $F_0(\alpha,\beta)$. So assume that $\beta \cdot \alpha^{-1} \in \mu$, $\alpha \notin S^1(K)$, and $\beta \notin S^1(K)$. Then $\xi_0:=\st(\alpha)=\st(\beta)$, where $\st$ is the standard part map computed on the circle. 
Since $K$ is countable, we can find $\xi_1 \in S^1(\R) \setminus \acl(K,\alpha,\beta,\xi_0)$, and put $\xi:=\xi_0\cdot \xi_1$. Then  $\xi \in S^1(\R) \setminus S^1(K)$, and since  $\alpha, \beta \notin S^1(K)$, the exchange property for $\acl$ implies that  $\alpha, \beta \notin \acl(K,\xi_1)$.
Define $a=\alpha \cdot \xi_1$, $b=\xi_1^{-1}$, $a_1=\beta \cdot \xi_1$, $b_1 = \xi_1^{-1}$. We check that $a,b,a_1,b_1$ witness that $F_0(\alpha,\beta)$. The equalities $\alpha =a \cdot b$, $\beta =a_1 \cdot b_1$, and $\tp(b/K)=\tp(b_1/K)$ are obvious. Since $\st(a) = \st(a_1) = \xi \notin S^1(K)$,  we get $\tp(a/K)=\tp(a_1/K)=p_\xi$. It remains to check that $\tp(a/K,\xi_1)$ and $\tp(a_1/K,\xi_1)$ are finitely satisfiable in $K$. As $\xi_1,\xi_2 \in S^1(\R)$ and $S^1(K)$ is dense in $S^1(\R)$, it is enough to check that $a,a_1 \notin \acl(K,\xi_1)$. But this is clear, as otherwise $\alpha \in \acl(K,\xi_1)$ or $\beta \in \acl(K,\xi_1)$, a contradiction.\\

By (\ref{equation: new in example}), $F_0$ is already an $M$-type-definable equivalence relation, so, by Theorem \ref{theorem: main theorem}, we conclude that $F_0=F={\bar F}=E'$.  This in turn implies that the classes of $E$ are the singletons $\{p_a\}$, $a \in S^1(\R) \setminus S^1(K)$, the singletons $\{q_a\}$, $a \in S^1(K)$, and the pairs $\{p_a^-,p_a^+\}$, $a \in S^1(K)$. Hence, $E$ is non-trivial, which implies that the ambit $(G,S_G(M),\tp(e/M))$ is not definable by Remark \ref{remark: uniqueness of E}.

One can check that $S_G(M)/E$ is the real circle $S^1$ with an additional copy of each point from $S^1(K)$, with the usual circle topology expanded by new subbasic open sets which are the singletons of the additional points of the circle  and their complements (in particular, each additional point is clopen).
Then there is a unique minimal left ideal $\M$ in $S_G(M)/E$ and it consists of $E$-classes of the non-algebraic types (which follows from the observations that the $G$-orbit of each algebraic type is dense and that each non-algebraic type lies in the closure of an arbitrary $G$-orbit). Moreover, there is a unique idempotent $u \in \M$, namely the $E$-class $\{p_1^-,p_1^+\}$. In particular, ${\mathcal M}=u{\mathcal M}$ by Fact \ref{Ellis theorem}(ii). One easily sees that $\theta \colon u\M \to G^*/{G^*}^{00}_M$ is an isomorphism. In contrast, $\hat{\theta}\colon S_G(M)/E \to G^*/{G^*}^{00}_M$ is not injective, as it glues $\{p_a^-,p_a^+\}$ with $\{q_a\}$ for every $a \in S^1(K)$. 

Note also that in this example the universal definable $G$-ambit is different (i.e. non-isomorphic) from the universal externally definable $G$-ambit, which follows from the more general observation that if $E$ in non-trivial, then the universal externally definable $G$-ambit is not definable. Indeed, if $S_{G,\ext}(M)$ was a definable $G$-ambit, then its homomorphic image $S_G(M)$ would be also definable, so $E$ would be trivial by Remark \ref{remark: uniqueness of E}. 
One can check that $S_{G,\ext}(M)$ can be identified with the collection of all points of $S^1(K)$ and all left and right cuts at all points of $S^1(\R)$ with the topology whose description is left as an exercise.

\end{ex}

\section{On some results from \cite{KrPi} in the $\Delta$-definable context}\label{section: results from KrPi}

Most of the main results of \cite{KrPi} are about connections of the Ellis group and the externally definable generalized Bohr compactification of $G$ with quotients of $G^*$ by connected components. It is very important in there that we have a natural epimorphism from the Ellis group of the universal externally definable $G$-ambit to $G^*/{G^*}^{000}_M$. As we saw in the previous section, in the definable context, we do not know whether such an epimorphism exists. However, we have the natural epimorphism $\theta$ from the Ellis group of the universal definable $G$-ambit to $G^*/{G^*}^{00}_M$. 

In this section, we formulate variants of some results from \cite{KrPi} in our $\Delta$-definable context (with $G^*/{G^*}^{00}_{\Delta,M}$ in place of ${G^*}/{G^*}^{000}_M$) whose proofs are obvious adaptations of the proofs from \cite{KrPi}, so they will be omitted. It would be really interesting, however, to get these kind of results with  $G^*/{G^*}^{000}_{\Delta,M}$ in place of $G^*/{G^*}^{00}_{\Delta,M}$, which would allow us to extend or strengthen some results from \cite{KrPi}, and, in the case of a countable language, maybe apply to get information on the Borel cardinality of ${G^*}^{00}_{\Delta,M}/{G^*}^{000}_{\Delta,M}$ and simplify the proofs from \cite{KrPiRz} for the quotient  ${G^*}^{00}_{M}/{G^*}^{000}_{M}$.

In the final part of this section, we analyze connections with the externally definable topological dynamics, and, using \cite[Theorem 5.7]{ChSi}, we obtain a variant of this result in the definable context.

The key role in \cite{KrPi} is played by the so-called $\tau$-topology introduced by Ellis. Basic theory related to this notion is described in \cite[Chapter IX]{Gl} for the Ellis group of the universal $G$-ambit $\beta G$, but it works similarly for the Ellis group of universal $G$-ambits in many other categories. In \cite[Section 2]{KrPi}, we described how to work with the universal externally definable $G$-ambit. In our new, $\Delta$-definable context, everything works analogously, so we will skip all the discussions, sending the reader to \cite{Gl} and \cite{KrPi} for details. Let us only recall the main definitions.

We take the notation as in previous sections. In particular, $\M$ is a minimal left ideal in $G^*/E'_\Delta$, and $u \in \M$ is an idempotent.

\begin{dfn}
For $A \subseteq G^*/E'_\Delta$ and $p\in G^*/E'_\Delta$, $p \circ A$ is defined as the set of all points $x \in G^*/E'_\Delta$ for which there exist nets $(x_i)$ in $A$ and $(g_i)$ in $G$ such that $\lim_i g_i=p$ (here by $g_i$ we mean $g_i/E'_\Delta$) and $\lim_i g_ix_i=x$.
\end{dfn}

\begin{dfn}\label{Def: tau topology}
For $A \subseteq u\M$, define $\cl_\tau (A) = (u \circ A) \cap u\M$.
\end{dfn}

 The proofs of 1.2-1.12 (except 1.12(2)) from \cite[Chapter IX]{Gl} go through (with some slight modifications) in our context. In particular, $\cl_\tau$ is a closure operator on subsets of $u\M$, and it induces the so-called {\em $\tau$-topology} on $u\M$. This topology is compact and $T_1$, and multiplication is continuous in each coordinate separately. It is easy to see that $p \circ A$ is always closed, and so the $\tau$-topology on $u\M$ is weaker than the topology inherited from $G^*/E'_\Delta$.
\begin{dfn}
$H(u{\M})$ is the intersection of the sets $\cl_\tau(V)$ with $V$ ranging over all $\tau$-neighborhoods of $u$ in the group $u\M$. 
\end{dfn}
Then $H(u\M)$ is a $\tau$-closed, normal subgroup of $u\M$, and $u\M/H(u\M)$ is a compact, Hausdorff group (see \cite[Chapter IX, Theorem 1.9]{Gl}).

 The notion of {\em generalized Bohr compactification} was introduced in \cite[Chapter VIII]{Gl}. It is recalled in \cite[Definition 1.23]{KrPi} in the externally definable context. In the $\Delta$-definable situation, we take the same definition, replacing the expression ``externally definable'' by ``$\Delta$-definable''.  The reader is referred to 1.13, 1.14, 1.21, 1.22, 1.23 from \cite{KrPi} for details.

It was proven in \cite[Chapter IX, Theorem 4.2]{Gl} that the generalized Bohr compactification of a discrete group $G$ equals $u\M/H(u\M)$ (everything computed in $\beta G$). In \cite[Theorem 2.5]{KrPi}, this was extended to the externally definable context.
Since the class of $\Delta$-definable $G$-flows is closed under taking both products and quotients by closed, $G$-invariant equivalence relations, the proof from \cite{KrPi} yields

\begin{thm}\label{theorem: gen. Bohr comp.}
$u\M/H(u\M)$ is the $\Delta$-definable generalized Bohr compactification of $G$.
\end{thm} 

The following fact is folklore in general topology, but we give a justification.

\begin{fct}\label{fact: from Engelking}
If $f\colon X \to Y$ is a continuous epimorphism, where $X$ is a second-countable (i.e. with a countable basis of open sets) space and $Y$ is a compact, Hausdorff space, then $Y$ is also second-countable.
\end{fct}

\begin{proof}
If $Z$ is a topological space, then a family ${\mathcal N}$ of subsets of $Z$ is said to be a {\em network} for $Z$ if for every $z \in Z$ and its open neighborhood $U$, there is $N \in {\mathcal N}$ with $z \in N \subseteq U$. The smallest possible cardinality of a network for $Z$ is called the {\em network weight} of $Z$ and is denoted by $nw(Z)$. For every space $Z$ we clearly have $nw(Z) \leq w(Z)$ (where $w(Z)$ is the {\em weight} of $Z$, i.e. the smallest cardinality of a basis of open sets). \cite[Theorem 3.1.19]{En} tells us that if $Z$ is compact, Hausdorff, then $nw(Z)=w(Z)$. 

Now, take a countable basis $\{B_i : i \in \omega\}$ of $X$.  It is clear, by the continuity of $f$, that $\{f[B_i] : i \in \omega \}$ is a network for $Y$. Hence, $w(Y) = nw(Y) \leq \aleph_0$.
\end{proof}

\begin{cor}
If both the language and the model $M$ are countable, then $u\M/H(u\M)$ is a Polish, compact group.
\end{cor}

\begin{proof}
We know that it is a compact, Hausdorff group, so it remains to show that it is metrizable. For this, it is enough to show that it is second-countable (see \cite[Theorem 4.2.8]{En}). 

We know that $S_{G,\Delta}(M)$ is second-countable, and so is $S_{G,\Delta}(M)/E_\Delta$ (by Fact \ref{fact: from Engelking} and the observations that $S_{G,\Delta}(M)/E_\Delta$ is compact, Hausdorff and is the image of $S_{G,\Delta}(M)$ under a continuous map). 

Since the $\tau$-topology on $u\M$ is weaker than the topology inherited on $u\M$ from  $S_{G,\Delta}(M)/E_\Delta$, we have that for $u\M$ equipped with this inherited topology (and $u\M / H(u\M)$ equipped with its usual quotient topology coming from the $\tau$-topology on $u\M$) the quotient function $u\M \to u\M / H(u\M)$ is continuous. But this inherited topology on $u\M$ has a countable basis (by the second paragraph of the proof). Thus, since  $u\M / H(u\M)$ is compact, Hausdorff, we get that it is second-countable by Fact \ref{fact: from Engelking}.
\end{proof}

This corollary shows an advantage of $u\M/H(u\M)$ computed in the definable category in comparison with the same object computed in the externally definable category (where it does not have to be metrizabe). In \cite{KrPi}, ${G^*}^{00}_M/{G^*}^{000}_M$ is presented as a quotient of a closed subgroup of the group $u\M/H(u\M)$ computed in the externally definable context, and in \cite{KrPiRz}, it was used to get new information on the Borel cardinality of ${G^*}^{00}_M/{G^*}^{000}_M$. If we were able to present ${G^*}^{00}_M/{G^*}^{000}_M$ as a quotient of closed subgroup of the group $u\M/H(u\M)$ computed in the definable context, by the above remark, we would be immediately within a nice descriptive set-theoretic setting, which could simplify some arguments from \cite{KrPiRz} (for the objects that we are considering now) and maybe lead to new results. But in this paper, we only describe connections between $u\M/H(u\M)$ and  $G^*/{G^*}^{00}_M$.

Proposition \ref{proposition: theta is onto} gives us the epimorphism $\hat{\theta} \colon G^*/E'_\Delta \to G^*/{G^*}^{00}_{\Delta,M}$ whose restriction $\theta$ to $u\M$ is also an epimorphism. Using the explicit definition of $\hat{\theta}$, one can adapt the proof of \cite[Theorem 0.1]{KrPi} to get the next theorem. In fact, the proof of (2) is even simpler now, because we do not use the $F_n$'s.

\begin{thm}\label{theorem: factorization by H(uM)}
 Suppose that $u\M$ is equipped with the $\tau$-topology and $u\M/H(u\M)$ -- with the induced quotient topology. Then:
\begin{enumerate}
\item $\theta$ is continuous,
\item $H(u\M) \leq \ker(\theta)$,
\item the formula $p*H(u\M) \mapsto \theta(p)$ yields a well-defined continuous epimorphism $\bar \theta$ from $uM/H(uM)$ to $G^*/{G^*}^{00}_{\Delta,M}$.
\end{enumerate}
In particular, we get the following sequence of continuous epimorphisms
\begin{equation}
			u\M\twoheadrightarrow u\M/H(u\M)\xtwoheadrightarrow{\bar \theta}{}{G^*}/{G^*}^{00}_{\Delta,M}.
		\end{equation}
\end{thm}

We say that $G$ is {\em $\Delta$-definably strongly amenable} if it has no non-trivial $\Delta$-definable {\em proximal} $G$-flows (i.e. flows in which any two points $x$ and $y$ are {\em proximal} which means that there exists a net $(g_i)$ in $G$ such that $\lim g_i x = \lim g_i y$). This extends the notion of strongly amenable group from \cite{Gl}. For example, \cite[Chapter II, Theorem 3.4]{Gl} tells us that all nilpotent groups are strongly amenable so also $\Delta$-definably strongly amenable.  Now, we significantly generalize Corollary 0.4 from \cite{KrPi} (by dropping the definability of types assumption and by extending the context to the local, $\Delta$-definable one). The same proof as the one from \cite{KrPi} works, once we use Proposition \ref{proposition: Delta-definable Bohr compactification} and (in the final part of the proof) the explicit formula for $*$ obtained in Proposition \ref{proposition: formula for *}.

\begin{thm}\label{theorem: strong amenability}
Suppose $G$ is $\Delta$-definably strongly amenable. Then the natural epimorphism $\bar \theta \colon u\M/H(uM) \to G^*/{G^*}^{00}_{\Delta,M}$ is an isomorphism.
\end{thm}

The theorem implies that for $\Delta$-definably strongly amenable groups, the $\Delta$-definable generalized Bohr compactification is isomorphic with the $\Delta$-definable Bohr compactification.

Let  us finish this section with a comparison of externally definable and definable objects.

Let $\M_{\ext}$ be a minimal left ideal of the universal externally definable $G$-flow $S_{G,{\ext}}(M)$, and let $u_{\ext}\in \M_{\ext}$ be an idempotent. Since each $\Delta$-definable $G$-flow is externally definable, there is a unique epimorphism $\hat{h}\colon (G,S_{G,\ext}(M),\tp_{\ext}(e/M)) \to (G,S_{G,\Delta}(M)/E_\Delta,\tp_\Delta(e/M)/E_\Delta)$. Then $\hat{h}$ is an epimorphism of semigroups. 
Put $\M:=\hat{h}[\M_{\ext}]$ and $u:=\hat{h}[u_{\ext}]$ (so far $\M$ and $u$ were chosen arbitrarily at the beginning, but now we define them in this particular way). We easily get that $\M$ is a minimal left ideal, $u \in \M$ is an idempotent, and $h:=\hat{h}|_{u_{\ext}\M_{\ext}}$ is an epimorphism from $u_{\ext} \M_{\ext}$ to $u\M$.

\begin{rem}
$h$ is continuous, where both Ellis groups are equipped with the $\tau$-topologies.
\end{rem}

\begin{proof}
Let $D \subseteq u\M$ be $\tau$-closed. Let $D'=h^{-1}[D]$. The goal is to show that $D'$ is $\tau$-closed. Take $p \in \cl_\tau(D')$. There are nets $(g_i)$ in $G$ and $(x_i)$ in $D'$ such that $\lim_i g_i =u_{\ext}$ and $\lim_i g_ix_i = p$. Then $\lim_i \hat{h}(g_i) = \hat{h}(u_{\ext})=u$ and $\lim_i \hat{h}(g_i) \hat{h}(x_i)=\lim_i \hat{h}(g_ix_i)=\hat{h}(p)$. Moreover, $\hat{h}(g_i) =g_i/E'_{\Delta}$ and $\hat{h}(x_i) \in D$. Therefore, $h(p)=\hat{h}(p) \in \cl_\tau(D)=D$, hence $p \in D'$.
\end{proof}

By this remark, we see that $h[H(u_{\ext}\M_{\ext})] \leq H(u\M)$. Therefore, $h$ induces a continuous epimorphism from $u_{\ext}\M_{\ext}/H(u_{\ext}\M_{\ext})$ to $u\M/H(u\M)$.

Let $N\succ M$ be an $|M|^+$-saturated elementary extension, and let $S_{G,M}(N)$ be the space of all types in $S_{G}(N)$ finitely satisfiable in $M$. Then $S_{G,\ext}(M)$ can be naturally identified with $S_{G,M}(N)$, which we will be using freely. Newelski was considering the epimorphism $\hat{\theta}_{\ext} \colon S_{G,M}(N) \to G^*/{G^*}^{00}_M$ given by $\hat{\theta}_{\ext}(\tp(a/N))=a/{G^*}^{00}_M$ and conjectured that $\theta_{ext}:=\hat{\theta}_{\ext}|_{u_{\ext}\M_{\ext}}\colon u_{\ext}\M_{\ext} \to {G^*}/{G^*}^{00}_M$ is an isomorphism (at least  in nice situations), e.g. see the comment after \cite[Proposition 4.4]{Ne1}. In general, such a conjecture is false, but it turned out to be true for definably amenable groups in NIP theories \cite[Theorem 5.6]{ChSi}. Using this result, we easily get that the same is true in our definable category.

\begin{cor}\label{corollary: from ChSi}
Assume we are in the definable case (i.e. $\Delta$ consists of all formulas of the appropriate form).
If $G$ is definable amenable and $T:=\Th(M)$ has NIP, then $\theta\colon u\M \to G^*/{G^*}^{00}_M$ is an isomorphism.
\end{cor}

\begin{proof}
Consider first the case as above, namely with  $\M:=\hat{h}[\M_{\ext}]$ and $u:=\hat{h}[u_{\ext}]$.
Note that $\hat{h}(\tp(a/N))= a/E'$ for any $\tp(a/N) \in S_{G,M}(N)$. Therefore, using the definitions of $\hat{\theta}$ and $\hat{\theta}_{\ext}$, we get $\hat{\theta}_{\ext}=\hat{\theta} \circ \hat{h}$. Hence, $\theta_{\ext}=\theta \circ h$. On the other hand, \cite[Theorem 5.6]{ChSi} tells us that $\theta_{\ext}$ is an isomorphism, and, by the above observations, we know that $h$ is an epimorphism. Therefore, $\theta$ is an isomorphism.

Now, consider an arbitrary minimal left ideal $\M$ and an idempotent $u \in \M$. Let $\M_0:=\hat{h}[\M_{\ext}]$ and $u_0:=\hat{h}[u_{\ext}]$. By \cite[Chapter I, Proposition 2.5]{Gl}, there is an idempotent $v \in \M$ such that $vu_0=u_0$ and $u_0  v=v$. 
Then $f \colon u\M \to u_0\M_0$ given by $f(x)=u_0vxu_0=vxu_0$ is a group isomorphism (even $\tau$-continuous). Indeed, $f(x)f(y)=vxu_0vyu_0=vxvyu_0=vxyu_0=f(xy)$ (the fact that $xv=x$ follows from the fact that $u \in \M=\M v$ (as $v \in \M$), $x \in \M=\M u$ and $v$ is an idempotent), so $f$ is a homomorphism; to see that it is an isomorphism, one should check, by similar computations, that $g \colon u_0\M_0 \to u\M$ given by $g(y)=uyv$ is the inverse of $f$.
 Let $\theta_0 \colon u_0\M_0 \to {G^*}/{G^*}^{00}_M$ be $\hat{\theta}|_{u_0\M_0}$. By Proposition \ref{proposition: theta is onto} and the definitions of $\theta$, $\theta_0$ and $f$, we get $\theta = \theta_0 \circ f$. Indeed, $\theta_0(f(x))=\hat{\theta}_0(vxu_0)=\hat{\theta}(v) \hat{\theta}(x) \hat{\theta}(u_0)=\hat{\theta}(x)=\theta(x)$.  By the first paragraph of the proof, $\theta_0$ is an isomorphism. Hence, we conclude that $\theta$ is an isomorphism, too.	
\end{proof}

Similarly to the externally definable case, also in the $\Delta$-definable category there is a general question about the impact of changing the ground model $M$ on the topological-dynamic invariants $u\M$ and $u\M/H(u\M)$. In particular, if we compute these invariants for a bigger model, does there exist epimorphisms to the corresponding objects for the smaller model? If we assume NIP, is $u\M/H(u\M)$ independent of the choice of $M$?

Another interesting direction concerns some weaker versions of the notion of definable [extremal] amenability that naturally arise in our $\Delta$-definable category, but we leave this for the future. 

\section{Topological-dynamic invariants as Polish structures}\label{section: Polish structures}

In \cite{Kr}, the following notion was introduced.

\begin{dfn}
A {\em Polish structure} is a pair $(G,X)$, where $G$ is a Polish group acting on a set $X$ so that the stabilizer of any singleton is a closed subgroup of $G$. We say that $(G,X)$ is {\em small} if for every $n\in \omega$ there are only countably many orbits on $X^n$ under the action of $G$.
\end{dfn}

In \cite{Kr}, it is assumed that the action of $G$ on $X$ is faithful, but this assumption is  purely cosmetic. The notion of {\em $nm$-independence} was introduced in \cite[Definition 2.2]{Kr}, and it was proven that it has some nice properties (as forking independence in stable or simple theories), but the existence of $nm$-independent extensions requires the assumption of smallness of the Polish structure in question. Then a counterpart of basic stability theory was developed for small Polish structures. In particular,  a counterpart of a superstable structure was introduced and called an {\em $nm$-stable} Polish structure.
The following, particular case of Polish structures was studied deeply in \cite{Kr} and \cite{KrWa}.

\begin{dfn}
i)  A {\em compact $G$-space} is pair $(G,X)$, where $G$ is a Polish group acting continuously on a compact, Hausdorff space $X$.\\
ii) A {\em compact $G$-group} is pair $(G,H)$, where $G$ is a Polish group acting continuously and by automorphisms on a compact, Hausdorff group $H$. 
\end{dfn} 

Various structural theorems on compact groups in the context of small Polish structures were proved in \cite{Kr} and \cite{KrWa}, e.g.

\begin{fct}\label{fact: nm-stability implies nilpotency}
If $(G,H)$ is a small, $nm$-stable compact $G$-group, then $H$ is nilpotent-by-finite.
\end{fct}

The main motivation to introduce Polish structures was to apply model-theoretic ideas to study purely topological objects. There is a variety of examples of classical small Polish structures, e.g. various compact metric spaces considered with the full group of homeomorphisms are always Polish structures which are often small. On the other hand, it would be interesting to use small Polish structures to get new results in pure model theory. A joint idea with Jan Dobrowolski is to view some spaces of types as Polish structures. For example, if $M$ is a countable first order structure, then $\aut(M)$ is naturally a Polish group and  $(\aut(M),S(M))$ becomes a Polish structure (note that the action of $\aut(M)$ on the type space $S(M)$ is continuous). However, even if $M$ is $\omega$-categorical (in particular, if its theory is small), this Polish structure is not necessarily small (e.g. for the random graph it is not small). However, one can formulate the following conjecture. Note before that a small, complete theory in a countable language has a unique (up to $\cong$) countable saturated model.

\begin{conj}\label{conjecture: chernikov}
Assume that $M$ is a countable, saturated model of a small, NIP theory in a countable language. Let $\Delta$ be a finite set of formulas without parameters. Then:\\
i) $(\aut(M),S(M))$ is a small Polish structure,\\ 
ii) $(\aut(M),S_{\Delta}(M))$ is a small Polish structure.
\end{conj}

As both pairs are Polish structures, only smallness requires a proof.
Clearly, (i) implies (ii). Artem Chernikov suggested that maybe some ideas from \cite{Sh} could be used to prove (ii). 

As we will see in a moment, this conjecture is very important for potential applications of small Polish structures to say something new about topological-dynamic invariants, but it is also interesting in its own right.\\  

In this section, we explain how to view various invariants as Polish structures. Take the context and notation from previous sections. We start from a corollary of Theorem \ref{theorem: main theorem}.

\begin{cor}
$E_\Delta$ is $\aut(M)$-invariant.
\end{cor}

\begin{proof}
This is equivalent to the statement that $E'_{\Delta}$ is $\aut(\C/\{ M\})$-invariant. By Theorem \ref{theorem: main theorem}, $E'_{\Delta}={\bar F}_\Delta$. Since ${\bar F}_\Delta$ is the finest type-definable over $M$ equivalence relation containing $F_\Delta$, it is enough to show that $F_\Delta$ is $\aut(\C/\{M\})$-invariant. Since $F_\Delta$ is the transitive closure of ${F_0}_\Delta$, this boils down to showing that ${F_0}_\Delta$ is $\aut(\C/\{ M\})$-invariant. But this follows immediately from the definition of ${F_0}_\Delta$, namely, items (1)-(3) from the definition of ${F_0}_\Delta$ are clearly preserved under all automorphisms which fix $M$ setwise. 
\end{proof}

Thus, $\aut(M)$ acts on $S_{G,\Delta}(M)/E_\Delta$ in the natural way, namely $f (p/E_\Delta):=f(p)/E_\Delta$; denote this action by $\Phi \colon \aut(M) \times S_{G,\Delta}(M)/E_\Delta \to S_{G,\Delta}(M)/E_\Delta$.

\begin{prop}\label{proposition: the action of aut(M) is continuous}
$\Phi$ is continuous.
\end{prop} 

\begin{proof}
A basic open set in $S_{G,\Delta}(M)/E_\Delta$ is of the form $U_\varphi:=\{ p/{E_\Delta} : [p]_{E_\Delta} \subseteq [\varphi]\}$ for a $\Delta$-formula $\varphi=\varphi(x,\bar m)$ with parameters $\bar m$ from $M$. Let $o(\bar m)$ be the orbit of $\bar m$ under $\aut(M)$. We compute 
$$\Phi^{-1}[U_\varphi]= \bigcup_{{\bar m}' \in o(\bar m)} \{ f \in \aut(M): f({\bar m}')=\bar m\} \times U_{\varphi(x,{\bar m}')},$$ 
which is clearly open in  $\aut(M) \times S_{G,\Delta}(M)/E_\Delta$. 
\end{proof}

\begin{prop}\label{proposition: the action preserves *}
The action $\Phi$ preserves $*$. 
\end{prop}

\begin{proof}
Consider any $f \in \aut(M)$ and $p,q \in S_{G,\Delta}(M)$. By Proposition \ref{proposition: formula for *}, 
$$p/E_\Delta * q/E_\Delta=\tp_{\Delta}(ab/M),$$ 
where $b \models q$ and $a$ satisfies a $\Delta$-coheir extension of $p$ over $M,b$. Let $\bar{f}$ be an extension of $f$ to an automorphism of $\C$. Then $f(p) = \tp_\Delta({\bar f}(a)/M)$, $f(q)=\tp_\Delta({\bar f}(b)/M)$, and $\tp_\Delta({\bar f}(a)/M,{\bar f}(b))$ is finitely satisfiable in $M$. Therefore,
$$\begin{array}{lll}
f(p/E_\Delta) * f(q/E_\Delta) &= & f(p)/E_\Delta * f(q)/E_\Delta=\tp_\Delta({\bar f} (a){\bar f} (b)/M)/E_\Delta\\
&= & \tp_\Delta ({\bar f}(ab)/M)/E_\Delta= f(\tp_\Delta(ab/M)/E_\Delta)\\
&=& f(p/E_\Delta * q/E_\Delta).
\end{array}$$
\end{proof}

Let $\M$ be a minimal left ideal in $S_{G,\Delta}(M)/E_\Delta$ and $u \in \M$ an idempotent. We will need the following observation \cite[Chapter IX, Lemma 1.5]{Gl}.

\begin{fct}\label{fact: lim i tau-lim}
If $(p_i)$ is a net in $u\M$ converging (in the usual topology on $\M$) to $p$, then $\tau\mbox{-}\lim_i p_i= up$.
\end{fct}

Let $\aut(M/u)$ be the stabilizer of $u$ under the action $\Phi$. By Proposition \ref{proposition: the action of aut(M) is continuous} and the fact that $S_{G,\Delta}(M)/E_\Delta$ is Hausdorff, $\aut(M/u)$ is a closed subgroup of $\aut(M)$.

\begin{prop}
The action $\Phi$ induces a $\tau$-continuous action of $\aut(M/u)$ on $u\M$.
\end{prop}

\begin{proof}
By Proposition \ref{proposition: the action preserves *} and the fact that $\M= S_{G,\Delta}(M)/E_\Delta *u$, we see that any automorphism $f \in \aut(M/u)$ fixes both $\M$ and $u\M$ setwise. Thus, the action $\Phi$ induces an action of $\aut(M/u)$ on $\M$, and further on $u\M$.
Now, we want to show the continuity of this action of $\aut(M/u)$ on $u\M$.

Consider arbitrary nets $(f_i)$ in $\aut(M/u)$ and $(p_i)$ in $u\M$ such that $\lim_i f_i = f \in \aut(M/u)$ and $\tau\mbox{-}\lim_i p_i = p \in u\M$. We need to show that $\tau\mbox{-}\lim_if_i(p_i) = f(p)$. For this it is enough to prove that any subnet $(f_k'(p_k'))$ of $(f_i(p_i))$ has a subnet which is $\tau$-convergent to $f(p)$. Hence, we see that it is enough to show that, for any data as in the first sentence of this paragraph, the net $(f_i(p_i))$ has a subnet which is $\tau$-convergent to $f(p)$.

By the compactness of $\M$, there is a subnet $(p_k')$ of $(p_i)$ converging, in the usual topology on $\M$, to some $p'$, i.e. $\lim_k p_k'=p'$. The corresponding subnet $(f_k')$ of $(f_i)$ still converges to $f$, i.e. $\lim_k f_k'=f$.

By Proposition \ref{proposition: the action of aut(M) is continuous}, $\lim_k f_k'(p_k')=f(p')$. Hence, by Fact \ref{fact: lim i tau-lim} and Proposition \ref{proposition: the action preserves *}, we get $\tau\mbox{-}\lim_k f_k'(p_k')=uf(p')=f(u)f(p')=f(up')$.

On the other hand, by Fact \ref{fact: lim i tau-lim} applied to the net $(p_k)$, we get $p=\tau\mbox{-}\lim_i p_i=\tau\mbox{-}\lim_k p_k'=up'$. 

By the last two paragraphs, $\tau\mbox{-}\lim_k f_k'(p_k')=f(p)$, and the proof is finished.
\end{proof}

\begin{cor}
$\aut(M/u)$ acts continuously on $u\M/H(u\M)$, i.e. on the $\Delta$-definable generalized Bohr compactification of $G$ (see Theorem \ref{theorem: gen. Bohr comp.}).
\end{cor}

Note that $\aut(M)$ also acts on $G^*/{G^*}^{00}_{\Delta,M}$. Namely, for $f \in \aut(M)$, take any $\bar f \in \aut(\C)$ extending $f$ and define $f \cdot (a/{G^*}^{00}_{\Delta,M}):= {\bar f}(a)/{G^*}^{00}_{\Delta,M}$. (The fact that this action is well-defined follows easily from the observation that ${G^*}^{00}_{\Delta,M}$ is invariant under $\aut(\C/\{ M\})$ and contains all $a^{-1}b$ for $a \equiv_M b$.) 
By a similar argument to the proof of Proposition \ref{proposition: the action of aut(M) is continuous}, one can show 

\begin{prop}
The action of $\aut(M)$ on $G^*/{G^*}^{00}_{\Delta,M}$ is continuous.
\end{prop}

By the above observations, we get

\begin{cor}\label{corollary: Polish structures}
Assume that the model $M$ is countable.
The following are Polish structures: $(\aut(M), S_{G,\Delta}(M))$, $(\aut(M), S_{G,\Delta}(M)/E_\Delta)$, $(\aut(M/u), u\M)$, $(\aut(M/u), u\M/H(u\M))$, and $(\aut(M), G^*/{G^*}^{00}_{\Delta,M})$. More precisely, all these are compact $G$-spaces (except the third one, which is not necessarily Hausdorff). Moreover, the second one is a compact $\aut(M)$-semigroup with left-continuous semigroup operation, the fourth one is a compact $\aut(M/u)$-group, and the last one is a compact $\aut(M)$-group.
\end{cor}

From now on, we always assume that $M$ is countable. In order to apply some knowledge on small Polish structures, first one would have to describe interesting classes of theories for which some of the above Polish structures are small. Conjecture \ref{conjecture: chernikov} may provide such classes. 

Using Proposition \ref{proposition: theta is onto}, we easily get
\begin{rem}
If $(\aut(M), S_{G,\Delta}(M))$ is small, then all other Polish structures from Corollary \ref{corollary: Polish structures} are small, too.
\end{rem}

By Theorem \ref{theorem: factorization by H(uM)}, we easily get

\begin{rem}
If $(\aut(M/u), u\M/H(u\M))$ is small, then $(\aut(M), G^*/{G^*}^{00}_{\Delta,M})$ is small as well.
\end{rem}

There are two kinds of possible applications of small Polish structures. First of all, \cite[Corollary 5.9]{Kr} tells us that small compact $G$-groups are profinite.

\begin{cor}
i) If $(\aut(M/u), u\M/H(u\M))$ is small, then $u\M/H(u\M)$ is a profinite group.\\
ii) If $(\aut(M), G^*/{G^*}^{00}_{\Delta,M})$ is small, then $G^*/{G^*}^{00}_{\Delta,M}$ is a profinite group.
\end{cor}

Secondly, we would like to describe the algebraic structure of $u\M/H(u\M)$ and $G^*/{G^*}^{00}_{\Delta,M}$, but for this we would have to know that the corresponding Polish structures are not only small, but also $nm$-stable (e.g. to apply Fact \ref{fact: nm-stability implies nilpotency}).

\begin{cor}
i) If $(\aut(M/u), u\M/H(u\M))$ is small and $nm$-stable, then $u\M/H(u\M)$ is nilpotent-by-finite.\\
ii) If $(\aut(M), G^*/{G^*}^{00}_{\Delta,M})$ is small and $nm$-stable, then $G^*/{G^*}^{00}_{\Delta,M}$ is nilpotent-by-finite.
\end{cor}

We finish with a discussion on NIP and stable situations, but before that we need to make one general observation. 

Recall that by ${G^*}^{00}$ we denote the smallest type-definable (over arbitrary parameters from $\C$) subgroup of $G^*$ of bounded index, if it exists. Note if ${G^*}^{00}$ exists, then it is type-definable over $\emptyset$, so ${G^*}^{00}={G^*}^{00}_\emptyset$. Therefore,  ${G^*}^{00}$ exists if and only if ${G^*}^{00}_A$ does not depend on the choice of the parameter set $A$.

\begin{rem}\label{rem:existence of G00}
If ${G^*}^{00}$ exists, then for any set $\Delta$ of formulas of the appropriate form, ${G^*}^{00}_{\Delta,M}$ does not depend on the choice of the model $M$ and it is type-definable over $\emptyset$; in fact,  ${G^*}^{00}_{\Delta,M}$ is the smallest $\Delta$-type-definable (over arbitrary parameters from $\C$) subgroup of $G^*$ of bounded index
\end{rem}

\begin{proof}
By the existence of ${G^*}^{00}$, there exists the smallest $\Delta$-type-definable (over arbitrary parameters from $\C$) subgroup of $G^*$ of bounded index, which we denote by ${G^*}^{00}_{\Delta}$. This component is clearly invariant under $\aut(\C)$, so it is type-definable over $\emptyset$ by a collection of formulas $\{\varphi_i(x): i \in I\}$ closed under (finite) conjunctions. We will show that ${G^*}^{00}_{\Delta}={G^*}^{00}_{\Delta,M}$ for any model $M \prec \C$.

The inclusion $(\subseteq)$ is clear. For the other inclusion it is enough to show that ${G^*}^{00}_{\Delta}$ is the intersection of a family of sets which are $\Delta$-definable over $M$. Consider any $i \in I$. We will be done if we show that there exist a $\Delta$-formula $\varphi(x)$ over $M$ and $j \in I$ such that $\varphi_j(G^*) \subseteq \varphi(G^*) \subseteq \varphi_i(G^*)$. By the definition of ${G^*}^{00}_{\Delta}$ and compactness, there is a $\Delta$-formula $\varphi^*(x)$ over $\C$ such that ${G^*}^{00}_{\Delta} \subseteq \varphi^*(G^*) \subseteq \varphi_i(G^*)$. By compactness, there is $j \in I$ such that $\varphi_j(G^*) \subseteq \varphi^*(G^*) \subseteq \varphi_i(G^*)$. Since $M \prec \C$, we can replace the parameters of $\varphi^*(x)$ by some parameters from $M$, obtaining a $\Delta$-formula $\varphi(x)$ over $M$ for which $\varphi_j(G^*) \subseteq \varphi(G^*) \subseteq \varphi_i(G^*)$.
\end{proof}


When the theory $T:=\Th(M)$ has NIP, we know that ${G^*}^{00}$ exists (see \cite{Sh-1} or \cite[Theorem 8.4]{Si}). By Remark \ref{rem:existence of G00}, this implies that ${G^*}^{00}_{\Delta,M}$ is type-definable over $\emptyset$. Thus, if we assume that the language of $T$ is countable, then  $G^*/{G^*}^{00}_{\Delta,M}$ is a compact, metrizable group and $\aut(\C)$ induces a compact group, say $\AUT$, acting continuously on  $G^*/{G^*}^{00}_{\Delta,M}$ as a group of automorphisms (see \cite[Lemma 3.11]{LaPi} and \cite[Fact 1.3]{KrNe}). So $(\AUT, G^*/{G^*}^{00}_{\Delta,M})$ is a compact structure interpretable in $T$ over $\emptyset$, according to \cite[Definition 1.3]{Kr-1}. It is very easy to see that if $T$ is small, then this compact structure is also small, and then \cite[Remark 2.1]{Kr-1} tells us that $G^*/{G^*}^{00}_{\Delta,M}$ is a profinite space (i.e. a totally disconnected, compact, Hausdorff space (see \cite[Theorem 1.1.12]{RiZa} for equivalent definitions of a profinite space)) which implies that it is a profinite group (see \cite[Theorem 2.1.3]{RiZa}). In fact, the assumption that the language is countable can be dropped in the last conclusion, as each small theory is a definitional extension of its reduct to a certain countable sublanguage. Alternatively, in the very simple proof of \cite[Remark 2.1]{Kr-1}, the fact that the underlying space  (in our case, $G^*/{G^*}^{00}_{\Delta,M}$) of the small compact structure in question is metrizable is irrelevant to conclude that it is profinite, hence the countability of the language can be dropped. So we have justified the following

\begin{rem}
If $T:=\Th(M)$ is small and has NIP, then $G^*/{G^*}^{00}_{\Delta,M}$ is a profinite group.
\end{rem}


Assume that $T:=\Th(M)$ is stable. We will be using fundamental knowledge on stability and stable groups (e.g. see \cite[Chapter 1]{Pi}). By \cite[Chapter 1, Lemma 2.2(i)]{Pi} and the shape of the formulas in $\Delta$, one easily gets that the $G$-ambit $S_{G,\Delta}(M)$ is $\Delta$-definable, so it is the universal $\Delta$-definable $G$-ambit and the relation $E_\Delta$ is trivial (by Corollary \ref{cor: definable ambit is a quotient} and Remark \ref{remark: uniqueness of E}). Since (by stability) there is a generic type in $S_{G,\Delta}(M)$, Corollary 1.9 of \cite{Ne1} implies that there is a unique minimal left ideal (equivalently, minimal subflow) $\M$ of $S_{G,\Delta}(M)$ and it consists of all the generic types. Another consequence of stability and the shape of the formulas in $\Delta$ is that any coset of ${G^*}^{00}_{\Delta,M}$ determines a unique generic type in $S_{G,\Delta}(M)$ (which is the $\Delta$-type over $M$ of some element of this coset). Together with Proposition \ref{proposition: formula for *}, this implies that there is a unique idempotent $u \in \M$ which is exactly the unique generic type containing all $\Delta$-formulas over $M$ defining ${G^*}^{00}_{\Delta,M}$. Thus, $\M=u\M$, $\theta \colon u\M \to G^*/{G^*}^{00}_{\Delta,M}$ is a topological isomorphism, the usual topology on $\M$ coincides with the $\tau$-topology on $u\M$, and $G^*/{G^*}^{00}_{\Delta,M}$ is profinite. 

By work of Newelski (e.g. see \cite[Proposition 1.6]{Ne-2} and \cite[Example 3]{Ne-1}), it follows that if the language is countable and $T$ is superstable with few countable models (so $T$ is small), then $(\AUT, G^*/{G^*}^{00}_{\Delta,M})$ is a small, $m$-stable profinite structure (in the sense of \cite{Ne-1}), which in turn implies, by \cite{Wa}, that $G^*/{G^*}^{00}_{\Delta,M}$ is abelian-by-finite.

\section*{Acknowledgments}
The author would like to thank the Referee for careful reading and comments  which helped to improve presentation.


\begin{thebibliography}{lll}
\bibitem{Au} J. Auslander, {\em Minimal flows and their extensions}, Mathematics Studies 153, North-Holland, The Netherlands, 1988.

\bibitem{ChSi} A. Chernikov, P. Simon. \emph{Definably amenable NIP groups}, submitted.

\bibitem{CoPi} A. Conversano, A. Pillay, \emph{Connected components of definable groups and $o$-minimality I},  Adv. Math. 231 (2012), 605-623.

\bibitem{El} R. Ellis, \emph{Lectures on topological dynamics}, W. A. Benjamin, New York, 1969.

\bibitem{En} R. Engelking, \emph{General topology}, Sigma Series in Pure Mathematics, Heldermann Verlag, Berlin, 1989.

\bibitem{Gi} J. Gismatullin,  \emph{Model theoretic connected components of groups},
Israel J. Math. 184 (2011), 251-274.

\bibitem{GiKr} J. Gismatullin, K. Krupiński, \emph{On model-theoretic connected components in some group extensions}, J. Math. Log. 15 (2015), 1550009 (51 pages). 

\bibitem{GiPePi} J. Gismatullin, D. Penazzi, A. Pillay, \emph{On compactifications and the topological dynamics of definable groups}, Ann. Pure Appl. Logic, 165 (2014), 552-562.

\bibitem{Gl} S. Glasner, \emph{Proximal Flows}, LNM 517, Springer, Germany, 1976.

\bibitem{Ja} G. Jagiella, \emph{Definable topological dynamics and real Lie groups}, Math. Log. Quart. 61 (2015), 45-55.

\bibitem{Kr-1} K. Krupi\'{n}ski, {\em Generalizations of small profinite structures}, J. Symbolic Logic 75 (2010), 1147-1175. 

\bibitem{Kr} K. Krupi\'nski, \emph{Some model theory of Polish structures}, Trans. Amer. Math. Soc. 362 (2010), 3499-3533.

\bibitem{KrNe} K. Krupi\'{n}ski, L. Newelski, \emph{On bounded type definable equivalence relations}, Notre Dame J. Form. Log. 43 (2002), 231-242.

\bibitem{KrPi}
K. Krupi\'{n}ski, A. Pillay, \emph{Generalized Bohr compactification and
model-theoretic connected components}, Math. Proc. Cambridge Philos. Soc., to appear. Published on-line, DOI: https://doi.org/10.1017/S0305004116000967.

\bibitem{KrPiRz} K. Krupi\'{n}ski, A. Pillay, T. Rzepecki, \emph{Topological dynamics and the complexity of strong types}, submitted.

\bibitem{KrWa} K. Krupi\'nski, F. Wagner, \emph{Small, nm-stable compact G-groups}, Israel J. Math. 194 (2013), 907-933. 

\bibitem{LaPi} D. Lascar, A.Pillay, \emph{Hyperimaginaries and automorphism groups}, J. Symbolic Logic 66 (2001), 127-143.

\bibitem{Ne-2} L. Newelski, \emph{$\M$-gap conjecture and $m$-normal theories}, Israel J. Math. 106 (1998), 285-311.
  
\bibitem{Ne-1} L. Newelski, \emph{Small profinite structures}, Trans. Amer. Math. Soc. 354 (2001), 925-943.

\bibitem{Ne1} L. Newelski, \emph{Topological dynamics of definable groups actions}, J. Symbolic Logic 74 (2009), 50-72.
  
\bibitem{Ne2} L. Newelski, \emph{Model theoretic aspects of the Ellis semigroup}, Israel J. Math. 190 (2012), 477-507.

\bibitem{Pi} A. Pillay, \emph{Geometric stability theory}, Clarendon Press, Oxford, 1996.

\bibitem{Pi1} A. Pillay, \emph{Type-definability, compact Lie groups, and $o$-minimality}, J. Math. Log. 4 (2004), 147-162.

\bibitem{RiZa} L. Ribes, P. Zalesskii, {\em Profinite groups}, volume 40, Springer, 2000.

\bibitem{Sh-1} S. Shelah, \emph{Minimal  bounded  index  subgroup  for  dependent  theories},
Proceedings of the American Mathematical Society 136 (2008), 1087-1091.

\bibitem{Sh} S. Shelah, \emph{Dependent dreams: recounting types}, preprint Sh:950.

\bibitem{Si} P. Simon, \emph{A guide to NIP theories (Lecture Notes in Logic)}, Cambridge University Press, 2015.

\bibitem{Wa} F. Wagner \emph{Small $m$-stable profinite groups}, Fund. Math. 176 (2003), 181-191.

\bibitem{YaLo} N. Yao, D. Long, \emph{Topological dynamics for groups definable in real closed field}, Ann. Pure Appl. Logic 166 (2015), 261–273.

\end{thebibliography}
\end{document}